\documentclass[a4paper,10pt]{article}
\usepackage[utf8]{inputenc}
\usepackage{geometry}
\usepackage{color}
\usepackage{graphicx}
\usepackage{amssymb}
\usepackage{epstopdf}
\usepackage{mathrsfs}
\usepackage{amsmath}
\usepackage{amsthm}
\usepackage{enumerate}
\usepackage[dvipsnames]{xcolor}
\usepackage{hyperref}
\newcommand\R{\mathbb{R}}
\renewcommand{\Im}{\operatorname{Im}}
\renewcommand{\Re}{\operatorname{Re}}
\newcommand\Z{\mathbb{Z}}
\newcommand\N{\mathbb{N}}
\newcommand\C{\mathbb{C}}
\newcommand\cK{\mathcal{K}}
\newcommand\E{\mathbb{E}}
\newcommand\Q{\mathbb{Q}}

\newcommand{\cQ}{\mathcal{Q}}

\newcommand{\cP}{\mathcal{P}}

\usepackage[most]{tcolorbox}

\begin{document}

\newcommand{\dd}{\mathrm{d}}
\newcommand{\ee}{\mathrm{e}}
\newcommand{\ii}{\mathrm{i}}
\newtheorem{theoreme}{Theorem}
\newtheorem{definition}{Definition}
\newtheorem{lemme}{Lemma}
\newtheorem{rem}{Remark}
\newtheorem{exemple}{Example}
\newtheorem{proposition}{Proposition}
\newtheorem{corollaire}{Corollary}
\newtheorem{hyp}{Hypothesis}
\newtheorem*{theo}{Theorem}
\newtheorem*{prop}{Proposition}
\newtheorem*{main*}{Main results}
\newtheorem*{LemStar}{Lemma}

\newcommand{\mar}[1]{{\marginpar{\sffamily{\scriptsize
        #1}}}}
\newcommand{\mi}[1]{{\mar{MI:#1}}}

\title{Scattering resonances of large weakly open quantum graphs}
\author{Maxime Ingremeau\footnote{Université Côte d'Azur, Laboratoire J.A. Dieudonné}}
\date{}

\maketitle

\abstract{In this paper, we consider a sequence of open quantum graphs, with uniformly bounded data, and we are interested in the asymptotic distribution of their scattering resonances. Supposing that the number of leads in our quantum graphs is small compared to the total number of edges, we show that most resonances are close to the real axis. More precisely, the asymptotic distribution of resonances of our open quantum graphs is the same as the asymptotic distribution of the square-root of the eigenvalues of the closed quantum graphs obtained by removing all the leads.}

\section{Introduction}

\subsection{Weyl's laws for scattering resonances}
When dealing with Schrödinger operators in $\R^d$, with a compactly supported potential, the spectrum of the operator consists of $[0, +\infty)$, along with finitely many negative eigenvalues.
Thus, a lot of relevant spectral information about such operators consist, not in the spectrum and the eigenfunctions themselves, but rather in the \emph{scattering resonances} and \emph{scattering states}. These states are eigenfunctions which are not square-integrable, but instead are \emph{outgoing} at infinity. The associated eigenvalues are the scattering resonances, which are complex numbers having the following interpretation: the real part is associated to the speed of oscillations, while the imaginary part corresponds to the escape rate of the scattering state towards infinity, when propagated by the wave equation.

For instance, in dimension 1, a number $z\in \C$ is a resonance of an operator $-\Delta +V$ if there exists a function $f\in C^\infty(\R)$ such that
$$\begin{cases}
-\Delta f + Vf = z^2 f\\
f(x) = a_\pm e^{\pm izx} ~~\text{ when $\pm x>>1$}. 
\end{cases}$$

We refer the reader to the recent monograph \cite{DZ} both for a friendly introduction to scattering resonances, and for a comprehensive presentation of this field of research.

 Scattering resonances can be seen as (square roots of) eigenvalues of a \emph{non-selfadjoint} operator. It is notorious that the spectrum of non-selfadjoint operators is much more difficult to understand than that of self-adjoint ones; this is mainly due to pseudospectral effects, where a small perturbation in the operator can produce large perturbations of the spectrum.
 
 Thus, it is complicated to determine the localization of the resonances. In dimension $d=1$, it is possible to count the number of resonances in large disks. Namely, if we denote by $N(R)$ the number of resonances in $D(0,R)$ counted with multiplicity, we have
 \begin{equation}\label{eq:WeylPot}
 \frac{N(R)}{R^d} \underset{R\to +\infty}{\longrightarrow} \frac{2 \mathrm{Vol}(\mathrm{chsupp} (V))}{\pi},
 \end{equation}
where $\mathrm{Vol}(\mathrm{chsupp} (V))$ denotes the length of the convex hull of the support of $V$.

Equation (\ref{eq:WeylPot}) was first proved for some special case of $V$ in \cite{Reg}, and then generalized in \cite{ZwoBound} (see also \cite{Froese} and \cite{SimRes}). It can be seen as an analogue of the classical \emph{Weyl's law} for the eigenvalues of the Laplacian in a bounded domain.

In higher dimension, it is not known if the analogue of (\ref{eq:WeylPot}) holds. All that is known in general is the upper bound proven in \cite{ZwoMultiD}, saying that $N(R)\leq C_V R^d$. 

Another interesting problem is to count the number of resonances not in disks, but in strips below the real axis. From a physical point of view, this is more relevant, as the resonances close to the real axis are related to resonant states which spend a long time in the system, before they escape to infinity. This question is often studied in the context of geometric scattering, or in the semiclassical framework. We refer to \cite{ZwoRes} for an account of the recent developments in this direction.

\subsection{Scattering resonances on quantum graphs}

As explained in the previous section, it is easier to count resonances of Schrödinger operators on $\R$ than on $\R^d$. A popular family of one-dimensional systems is that of \emph{quantum graphs}. A quantum graph is a graph where each edge is endowed with a length, and each vertex, with a boundary condition. Eigenfunctions on the quantum graph are eigenfunctions of the Laplacian on each edge (and hence, are complex exponentials), satisfying the required boundary conditions at each vertex. The most natural boundary conditions are the \emph{Kirchhoff} (or \emph{Neumann-Kirchhoff}) \emph{boundary conditions} (see Definition \ref{def:Kirch}); in all this paper, we will be working with these conditions.

Though very simple to describe, quantum graphs and their spectrum have been successfully used as a model of quantum chaos since the pioneering work of Kottos and Smilansky \cite{KS97}. Here, we are interested in quantum graphs where, to some of the vertices, we attach semi-infinite edges (or \emph{leads}). Scattering resonances on such open quantum graphs are complex numbers $z$ such that there exists an eigenfunction satisfying the boundary conditions at each vertex, and equal to a constant times $e^{i z x}$ on every lead. Resonances can also be seen as poles of the scattering matrix or as poles of the scattering determinant (see \cite{Equivalence}).

\paragraph{Counting resonances in disks}

It is natural to wonder if (\ref{eq:WeylPot}) holds on open quantum graphs. A natural candidate to replace the constant $\mathrm{Vol}(\mathrm{chsupp} (V))$ would be the length of the (finite part of the) quantum graph, i.e., the sum of the lengths of all the edges that are not leads.

Actually, there are examples where $N(R)$ behaves asymptotically like a constant times $R$, where the constant is not $\frac{2}{\pi}$ times the length of the graph. Such an open quantum graph is sometimes called a \emph{non-Weyl graph}. In  \cite{NonWeyl}, it was shown that, on an open quantum graph with Kirchhoff boundary conditions at every vertex, we have $\frac{N(R)}{R}\longrightarrow c$, where $c$ is $\frac{2}{\pi}$ times the length of the graph if and only if the graph is \emph{unbalanced}, i.e., at every vertex, the number of internal edges is different from the number of leads.

The simplest example of a non-Weyl (thus, not unbalanced) quantum graph is the following: if we have a vertex to which one internal edge and one lead are attached, the situation is equivalent to the one we obtain by removing the vertex and the lead, and replacing the edge by a lead; this would thus reduce the length of the quantum graph. Other examples of non-Weyl graphs were constructed in \cite{NonWeyl}, \cite{NonWeylEffective}, \cite{NonWeylMag}, \cite{NonWeyGen}. The article \cite{Lip} provides a nice presentation of all these results.

\paragraph{Counting resonances in strips}

Just as for Schrödinger operators on $\R^d$, it is interesting to count the number of resonances in strips below the real axis, instead of disks. In \cite{NonWeyl}, it is shown that, for quantum graphs satisfying Kirchhoff boundary conditions, the resonances all lie in a strip of the form $\Im z\in [-K, 0]$. To our knowledge, no explicit expression has ever been given for $K$; we will do so in Lemma \ref{lem:InABand} below in the case of unbalanced quantum graphs.

It is natural to wonder if we can also have a non-trivial \emph{upper bound} on the imaginary part of resonances, that is to say, if they do all lie in a strip of the form $\Im z \in [-K, - K']$ for some $K'>0$.

This cannot hold in general, as there can be resonances on the real axis. This is something typical of quantum graphs, related to the fact that, on a quantum graph, there can exist eigenfunctions which vanish on a large part of the graph. For instance, if a quantum graph contains a cycle whose edges have lengths that are all commensurate, it is possible to build eigenfunctions which will vanish on every vertex of the cycle, and on all the edges that are not part of the cycle (see Figure \ref{exampleTrivial} in the appendix). This is in strict opposition to what happens for Schrödinger operators in bounded domains or on manifolds, where the unique continuity principle forbids an eigenfunction to vanish on an open set.

Actually, this situation is highly non-typical: it is shown in \cite{Rational} and \cite{FermiGoldenRule} that, if a quantum graph has a resonance on the real axis, then a small perturbation of the lengths will turn it into a resonance with negative imaginary part. Also, it is shown in  \cite[Theorem 4.2]{CdvT} that, for almost all choices of edge lengths, there are no resonances on the real axis.
So far, the only known examples of quantum graphs having resonances on the real axis either have some loop, some vertices of degree 2, or have the lengths of the edges on a cycle all commensurate (see Figure \ref{exampleTrivial} in Appendix \ref{app:Irr}). It is thus natural to think that, when all edge lengths are rationally independent, all degrees are at least 3, and the graph contains no loops, there cannot be any resonance on the real axis. We will show in Appendix \ref{app:Irr} that this is not the case: we show (in a non-constructive way) the existence of quantum graphs containing no loops, having minimal degree larger than 3, and with rationally independent edge lengths, but having some resonances on the real axis.

Coming back to the existence of a resonance gap, i.e., to the question of whether resonances lie in a strip of the form $\Im z \in [-K, - K']$ for some $K'>0$, the analogy with semiclassical Schrödinger operators may lead us to look for criteria that are less related to arithmetic of the edge lengths, and more to some classical dynamics. In the physics literature, this was proposed in \cite{BarGas}, where the authors introduce some classical dynamics on a quantum graph, introduce a topological pressure\footnote{analogue to the topological pressure of half the unstable Jacobian for semiclassical Schrödinger operators whose classical dynamics have a hyperbolic trapped set. See \cite{NonnenReview} for a definition of this quantity in the case of semiclassical Schrödinger operators.}, and claim that, if this pressure is negative, there should be a resonance gap.

This statement was disproved in  \cite{GSS} thanks to physical and numerical arguments, and in \cite{CdvT} by mathematical proofs. Indeed, in \cite{CdvT}, the authors show that, for quantum graphs that are not trees, and have rationally independent edge lengths, there exist resonances arbitrarily close to the real axis (and with arbitrarily large real parts). The authors are even able to give an interesting lower bound on the number of resonances close to the real axis.

Finally, let us mention that, along with scattering resonances on an open quantum graph, many other non-selfadjoint problems have been considered on quantum graphs, like the damped wave equation, Schrödinger operators with complex potentials, or non-selfadjoint conditions at the vertices. We refer the reader to the introduction of \cite{RivRoy} for a presentation of the advances on such topics.

\subsection{Resonances of large quantum graphs}\label{sec1.3}

In this paper, we will not be interested in the resonances of a single quantum graph. We will rather consider a sequence of open quantum graphs $\mathcal{Q}_N$, having a number of vertices and edges of the order of $N$, and will consider the asymptotic distribution of their resonances. Namely the kind of question we will be interested in is the following:

\begin{equation}\label{Quest}
\begin{aligned}
\text{Given a region $\Omega\subset \C$, what is the asymptotic}
\\ \text{number of resonances of $\mathcal{Q}_N$ in $\Omega$ as $N\longrightarrow +\infty$ ?}
\end{aligned}
\end{equation}
Hence, we consider here the \emph{large graph limits} instead of the \emph{large frequency limits}.

To our knowledge, this is the first time in the mathematical literature that resonances of large open quantum graphs are considered. However, for closed quantum graphs (i.e., quantum graphs without leads, which is the situation the most frequently considered), the large graph limits has been a subject of growing interest in the past few years. An important motivation for that was the question of Quantum Ergodicity, which is \textcolor{black}{a} property of equidistribution of the eigenfunctions in systems enjoying a chaotic (or just ergodic) classical dynamics. Although quantum graphs are a popular model of quantum chaos, it was shown in \cite{CdV15} that, in general, Quantum Ergodicity does not hold on a given quantum graph. It thus became natural to consider eigenfunctions on a \emph{sequence} of quantum graphs of increasing size. See for instance \cite{BKS07}, \cite{GKP10}, \cite{KaSc14} \cite{BrWi16}, \cite{QEQGEQ} and \cite{QEQG} for examples of families of quantum graphs for which quantum ergodicity holds, and \cite{BKW04} for an example where it doesn't.

 Let us also mention \cite{BSQG}, where the question (\ref{Quest}) is addressed for closed quantum graphs (hence, the resonances are just  the square roots of the usual eigenvalues of selfadjoint operators, which are thus on the real axis). 
Instead of counting exactly the number of resonances in a set $\Omega$, it is more natural to consider the \emph{empirical spectral measures}
$$\mu_{\mathcal{Q_N}}:= \frac{1}{\mathcal{L}(\cQ_N)} \sum_{z\in \mathrm{Res}(\mathcal{Q}_N)} \delta_{z},$$
where $\mathcal{L}(\cQ_N)$ is the sum of the lengths of the edges of $\mathcal{Q}_N$, $\mathrm{Res}(\mathcal{Q}_N)$ is the set of resonances of $\mathcal{Q}_N$. It is thus a locally finite measure, which is supported on $\R$ in the case of closed graphs.

\paragraph{Benjamini-Schramm convergence of quantum graphs}

In \cite{BSQG}, the authors show that a natural criterion to obtain convergence of the empirical spectral measures is to suppose \emph{Benjamini-Schramm convergence} (also dubbed \emph{local weak convergence}) of the quantum graphs. 
For a sequence of discrete graphs $(G_N)$, we say that it converges in the sense of Benjamini-Schramm (\cite{BS,AL}) if for any finite rooted graph $(F,o)$ and $k\in\N$, the fraction of points $x\in G_N$ such that $B_{G_N}(x,k) \cong (F,o)$ has a limit, where $B_{G_N}(x,k)$ is the ball of radius $k$ around $x$ in the graph $G_N$. The limit is then a \emph{probability measure on the set of rooted graphs}.

This definition of Benjamini-Schramm convergence for discrete graphs can be extended to the case quantum graphs; we will recall how to do so in section \ref{sec:BS}. We will also recall that, if a sequence of quantum graphs has uniformly bounded degrees and edge lengths, we may always extract a subsequence which converges in the sense of Benjamini-Schramm.

The main result of \cite{BSQG} can be summed up as follows: If a sequence of closed quantum graphs $(\mathcal{Q}_N)$ has uniformly bounded degrees and edge lengths, and converges in the sense of Benjamini-Schramm to some probability measure $\mathbb{P}$ on the set of rooted quantum graphs, then  $\mu_{\mathcal{Q}_N}$ converges vaguely to some measure depending only on $\mathbb{P}$. More precisely, for any continuous compactly supported function $\chi \in C_c(\R)$, we have\footnote{Note that, in \cite{BSQG}, the authors define the empirical spectral measures on closed quantum graphs as sums of Dirac masses on the eigenvalues, which are squares of resonances. This is why the function we integrate in (\ref{eq:BSQG}) is $\chi(\sqrt{z})$.}

\begin{equation}\label{eq:BSQG}
\frac{1}{\mathcal{L}(\cQ_N)} \sum_{x\in \mathrm{Res}(\mathcal{Q}_N)} \chi(x) \underset{N\to +\infty}{\longrightarrow} \int_\R \chi(\sqrt{x}) \mathrm{d}\mu_{\mathbb{P}}(x) =: \int_{\mathbf{Q}_*} \chi\big{(}\sqrt{H_\cQ}\big{)}(\mathbf{x_0},\mathbf{x_0}) \dd \mathbb{P}(\cQ,\mathbf{x_0})
\end{equation}
where $H_\cQ$ is the Schr\"odinger operator on the limiting random quantum graph $\cQ$, $\chi\big{(}\sqrt{H_\cQ}\big{)}(\mathbf{x_0},\mathbf{x_0})$ is the value of the Schwartz kernel of the operator $\chi\big{(}\sqrt{H_\cQ}\big{)}$ at the root $\mathbf{x_0}$ and $\mathbf{Q}_*$ is a set of (equivalence classes of) rooted
quantum graphs.

\paragraph{Presentation of our main result}

The main result of this paper is a generalization\footnote{Note however, that we only consider case of Kirchhoff boundary conditions here, while more general boundary conditions are treated in \cite{BSQG}. Actually, the proofs of \cite{BSQG} are much simpler when only Kirchhoff boundary conditions are considered.} of (\ref{eq:BSQG}) to the case of a sequence of open quantum graphs converging, in the sense of Benjamini-Schramm, to a probability measure supported on the set of rooted \textbf{closed} quantum graphs:

\begin{tcolorbox}
Let $(\mathcal{Q}_N)$ be a sequence of \textcolor{black}{unbalanced} open quantum graphs  with $N$ vertices, having uniformly bounded degree and edge lengths. Suppose that $(\mathcal{Q}_N)$ converges, in the sense of Benjamini-Schramm, to a measure probability $\mathbb{P}$ supported on the set of rooted closed quantum graphs.
Then $\mathcal{Q}_N$ satisfies (\ref{eq:BSQG2}). In particular, $\mu_{\mathcal{Q}_N}$ converges vaguely to some measure on $\R$ depending only on $\mathbb{P}$.
\end{tcolorbox}

Recall that, up to extracting a subsequence, we may always suppose that $(\cQ_N)$ converges in the sense of Benjamini-Schramm. The assumption that the limit is on the set of closed quantum graphs is equivalent to supposing that the number of leads is a $o(N)$.

An informal, but illuminating, formulation of our result is the following. Let $\widetilde{\mathcal{Q}}_N$ be the quantum graph obtained from $\mathcal{Q}_N$ by removing all the leads from it. Then, if the number of leads of $\cQ_N$ is a $o(N)$, the resonances of $\cQ_N$ and the square-root of the eigenvalues of $\widetilde{\cQ}_N$ are asymptotically distributed the in same way.

In particular, this shows that, if $\cK$ is a compact set which does not intersect the real axis, then the number of resonances of $\cQ_N$ in $\cK$ is a $o(N)$.

\paragraph{Resonances mildly close to the real axis}

In Theorem \ref{th:FewRes} below, we will show more precise estimates, depending on the number of leads in $\cQ_N$. In particular, if the number of leads in $\cQ_N$ is bounded independently of $N$, we show that the number of resonances $\cQ_N$ in $[a,b] + i[-\infty, - \delta_N]$ is a $O\left(\frac{1}{\delta_N^{5/2}}\right)$. We don't believe this bound is optimal; it would be very interesting (and probably hard) to describe the asymptotic number of resonances $\cQ_N$ in $[a,b] + i[-\infty, -\delta_N]$, for instance when $\delta_N$ is a negative power of $N$. To this end, the numerical computations done in \cite[Section 7]{BarGas} could be a good starting point, although their $N$ isn't very large, and the asymptotic they consider is not very clear (they seem to consider the asymptotic $\delta_N >>1$, although there are no resonances away from a horizontal strip).

\paragraph{Relation to other works}

The large quantum graph limit we consider (with few leads) has analogues in the realm of Schrödinger operators. Namely, consider a potential $V\in L^\infty(\R)$, spatially extended (for instance, periodic with respect to $\Z^d$, or a random potential), and write $V_R(x):= V(x) \mathbf{1}_{|x|\leq R}$. It is natural to wonder if the resonances of $-\Delta + V_R$ will somehow converge to the real axis, and approach the spectrum of the fully extended operator $-\Delta+V$. The reason why such a phenomenon would happen is that, to escape from the support of the potential and go to infinity, the waves must pass through a spherical layer close to the boundary of the ball. As the ball becomes large, this layer has a volume much smaller than the total volume of the ball. Hence, this situation is very similar to having few leads on a large quantum graph.

To our knowledge, this problem of resonances for potentials with a wide support was first studied in \cite{Kl} for both periodic and random potentials in $\Z$, which are truncated outside of large intervals.  The author obtains very precise estimates on the localization of the resonances for both situations he considers. We cannot hope to obtain such precise results here, as we are working in the framework of Benjamini-Schramm convergence, which encompasses situations as different as those of periodic and random potentials. We refer the reader to \cite{DrouotThese} for a review of all the results involving resonances of potentials with a wide support.

Finally, we would like to mention that the methods we use here are very much inspired from those developed by J. Sjöstrand to study scattering resonances or the spectrum of the damped wave equation, as in \cite{Sj2} or \cite{Sj1}.

\paragraph{Organisation of the paper}
After this long introduction, we will start in section \ref{sec:General} by stating our results in a more general setting involving holomorphic families of large matrices. In section \ref{sec:QG}, we will recall the definition of quantum graphs and of their resonances, and explain how they enter the framework developed in the previous section. In section \ref{sec:Complex}, we will recall the tools of complex analysis we will use. Section \ref{sec:ProofFew} shows that there are few resonances away from the real axis, while section \ref{sec:ProofMain} gives a proof of our main result. Finally, Appendix \ref{app:Irr}, which is independent from the rest of the paper,  gives new examples of quantum graphs having scattering resonances on the real axis.

\paragraph{Acknowledgements}
It is a real pleasure to thank Frédéric Klopp for telling me about the problem of resonances for potentials with a wide support, which was the starting point of the reflection which led to this article. I also benefited a lot from discussions with Guillaume Klein and Martin Vogel on related topics.

Last but not least, I would like the thank the anonymous referee for indicating several interesting references, and pointing out a few imprecisions in the manuscript.

\section{A general framework for our results}\label{sec:General}

We would like to state our results in a general framework of holomorphic families of operators acting on some Hilbert spaces. In most of the paper, we will restrict ourselves to the finite-dimensional case, which is the most relevant for quantum graphs.

Let $Y_0\in (0, +\infty]$. We will write $\C^{Y_0}:= \{z\in \C ; \Im z <Y_0\}$.

For each $N\in \N$, let $\mathcal{H}_N$ be a Hilbert space, and let $U_N(z) : \mathcal{H}_N \longrightarrow \mathcal{H}_N$ be a sequence of trace-class operators depending in a holomorphic way on $z\in \C^{Y_0}$.
Recall that, if $U$ is a trace class-operator with trace-norm $\|U\|_1$, the determinant of $\mathrm{Id}+U$ is well-defined, and we have
\begin{equation}\label{eq:PropSchatten}
|\mathrm{\det} (\mathrm{Id} + U)| \leq e^{\|U\|_1}.
\end{equation}

Recall furthermore that, if $U$ is bounded and $V$ is trace-class, we have
\begin{equation}\label{eq:PropSchatten1}
\begin{aligned}
\|UV \|_1 &\leq \|U\| \|V\|_1\\
\|VU \|_1 &\leq \|U\| \|V\|_1.
\end{aligned}
\end{equation}

We will be interested in solutions of the equation
\begin{equation}\label{eq:DetNul}
\det (U_N(z) - \mathrm{Id})=0.
\end{equation}

A solution of (\ref{eq:DetNul}) will be called a \emph{resonance} of $U_N$. If $z$ is a solution of (\ref{eq:DetNul}), $\ker (U_N(z) - \mathrm{Id})$ is a vector space of positive finite dimension. This dimension will be called the \emph{multiplicity} of the solution of (\ref{eq:DetNul}).

If $\Omega\subset \C^{Y_0}$ is an open set, we will write
\begin{equation*}
\mathcal{N}_\Omega (U_N) := \text{number of solutions of (\ref{eq:DetNul}) in $\Omega$, counted with multiplicity}.
\end{equation*}

Our aim will be to estimate the asymptotic behaviour of $\mathcal{N}_\Omega (U_N) $ as $N\to +\infty$. To this end, we make the following assumption on $U_N$.

\begin{tcolorbox}

\begin{hyp}\label{Hyp:GenRes}
\begin{enumerate}
\item For any $N\in \N$, $U_N$ has no resonances in $\{z\in \C; 0<\Im z < Y_0\}$.

\item If $\mathcal{K}\subset \{z\in \C ; 0< \Im z < \textcolor{black}{Y_0}\}$ is a compact set, we may find a constant $\mathrm{C}_0(\mathcal{K})>0$, independent of $N$, such that for all $N\in \N$ and all $z\in \mathcal{K}$, we have 
\begin{equation}\label{eq:HypResolv0}
\left\| \left(\mathrm{Id}- U_N(z)\right)^{-1}\right\|\leq \frac{\mathrm{C_0}(\mathcal{K})}{|\Im z|}.
\end{equation}

\item \textcolor{black}{There exists $Y_1 >0$ such that, for any $N\in \N$, $U_N$ has no resonances in $\{z\in \C; \Im z \leq -Y_1\}$. Furthermore,  if $\mathcal{K}\subset \{z\in \C ; \Im z \leq - Y_1\}$ is a compact set, we may find a constant $\mathrm{C}'_0(\mathcal{K})>0$, independent of $N$, such that for all $N\in \N$ and all $z\in \mathcal{K}$, we have 
\begin{equation}\label{eq:HypResolv2}
\left\| \left(\mathrm{Id}- U_N(z)\right)^{-1}\right\|\leq C_0'(\mathcal{K}).
\end{equation}}

\item Let $\mathcal{K}\subset \C^{Y_0}$ be a compact set. There exists $\mathrm{C_1(\mathcal{K})}>0$ such that for all $N\in \N$ and all $z\in \mathcal{K}$, we have
\begin{equation}\label{eq:BorneGlobaleU}
\left\| U_N(z) \right\|\leq \mathrm{C_1}(\mathcal{K}).
\end{equation}
\end{enumerate}
\end{hyp}
\end{tcolorbox}
Note that, by analyticity of $U_N(z)$, equation (\ref{eq:BorneGlobaleU}) implies that there exists a constant 
$\mathrm{C_1'(\mathcal{K})}>0$ such that for all $N\in \N$ and all $z\in \mathcal{K}$, we have
\begin{equation}\label{eq:BorneGlobaleDerivU}
\left\| \frac{\mathrm{d}U_N(z)}{\mathrm{d}z} \right\|\leq \mathrm{C'_1}(\mathcal{K}).
\end{equation}

Our main assumption is that $U_N$ is close, in trace-norm, to an operator having all its resonances on the real axis. In the case of quantum graphs, these new operators will be built from the quantum graphs where all the leads are removed.
\begin{tcolorbox}
\begin{hyp}\label{Hyp:TildeU}
We suppose that, for each $N\in \N$, there exists another family $\widetilde{U}_N(z) : \mathcal{H}_N \longrightarrow \mathcal{H}_N$ \textcolor{black}{of invertible operators} depending in a holomorphic way on $z\in \C^{Y_0}$, such that the following holds.
\begin{enumerate}
\item  \textcolor{black}{If $\mathcal{K}\subset \C^{Y_0}$ is a compact set, there exists a constant $c(\mathcal{K})$ such that the following holds for all $N\in \N$ and all $z\in \mathcal{K}$
\begin{equation}\label{eq:BorneSurLesNormes}
\begin{aligned}
\Im z >0  &\Longrightarrow \|\widetilde{U}_N(z)\| \leq \frac{1}{1- c(\mathcal{K}) \Im z}\\
\Im z <0  &\Longrightarrow \|\widetilde{U}_N^{-1}(z)\| \leq \frac{1}{1- c(\mathcal{K}) |\Im z|}.
\end{aligned}
\end{equation}}

\textcolor{black}{In particular, for all compact set $\mathcal{K}\subset \C^{Y_0} \setminus \R$, there exists $\mathrm{C}_2(\mathcal{K})$ such that for all $N\in \N$ and all $z\in \mathcal{K}$, we have}

\begin{equation}\label{eq:HypResolv}
\left\| \left(\mathrm{Id}- \widetilde{U}_N(z)\right)^{-1}\right\|\leq \frac{\mathrm{C_2(\mathcal{K})}}{|\Im z|}.
\end{equation}
\item There exists a sequence $(\gamma_N)$ of positive numbers such that
 the following holds.
For any compact set $\mathcal{K}\subset \C^{Y_0}$, we may find a constant $C(\mathcal{K})$ independent of $N$ such that
\begin{equation}\label{eq:DefDiffTrace}
\forall z\in \mathcal{K},~~ \left\| U_N(z) - \widetilde{U}_N(z)\right\|_1 \leq C(\mathcal{K}) \gamma_N.
\end{equation}
\end{enumerate}
\end{hyp}
\end{tcolorbox}

Note that (\ref{eq:HypResolv}) implies that the resonances of $\widetilde{U}_N(z)$ are all on the real axis. In the next section, if we are working with a quantum graph having $N$ vertices, $\gamma_N$ will be the number of leads, which should be small compared to $N$.

Our aim will be to count resonances in boxes. If $a<b$, $c<d$, we will write
$$\mathcal{R}_{a,b,c,d} := \{z\in \C ~|~ a\leq \Re z \leq b \text{ and } c\leq \Im z \leq d\}$$
and
$$\mathcal{N}_{a,b,c,d}(U_N) := \mathcal{N}_{\mathcal{R}_{a,b,c,d}}(U_N).$$

If $a<b$, we will also write
$$\mathcal{N}_{a,b}(\widetilde{U}_N) := \text{ number of resonances of $\widetilde{U}_N$ in $[a,b]$ counted with multiplicity}.$$

Our first theorem gives a bound on the number of resonances of $U_N$ away from the real axis. 

\begin{tcolorbox}[breakable, enhanced]
\begin{theoreme}\label{th:FewRes2}
Let $(U_N)$ and $(\widetilde{U}_N)$ satisfy Hypotheses \ref{Hyp:GenRes}, \ref{Hyp:TildeU}. Let $a_1< a_2 \in \R$, and let $a_3>0$. There exists a constant $C$ depending on $a_1, a_2, a_3$ and on the constants in Hypotheses \ref{Hyp:GenRes} and \ref{Hyp:TildeU}, but not on $N$, such that the following holds.

Let $(\delta_N)$ be a bounded sequence of positive numbers, possibly going to zero. Then there exists $N_0$ such that for all $N\geq N_0$,
\begin{equation}\label{eq:RienLoinDuReel}
\mathcal{N}_{a_1,a_2, - a_3, - \delta_N} (U_N)\leq \frac{C \gamma_N}{\delta_N^{5/2}}.
\end{equation}

In particular, if $\gamma_N = o(N)$ and $\Omega\subset \C$ is a bounded set such that $\overline{\Omega}\cap \R = \emptyset$, we have
\begin{equation}\label{eq:RienLoinDuReel2}
 \frac{1}{|N|}\mathcal{N}_{\Omega}(U_N)\underset{N\to \infty}{\longrightarrow} 0.
 \end{equation}
\end{theoreme}
\end{tcolorbox}

Equation (\ref{eq:RienLoinDuReel2}) follows from (\ref{eq:RienLoinDuReel}) by taking a sequence $(\delta_N)$ converging to zero slowly enough.

Note that equation (\ref{eq:RienLoinDuReel}) is non-trivial only when $\delta_N$ is not chosen too small. 

\begin{rem}
We will see in Theorem \ref{th:ResClose} that, if the dimension of $\mathcal{H}_N$ grows like $N$, then, under reasonable assumptions, $\mathcal{N}_{a_1,a_2, - a_3, 0} (U_N)$ is of the same order as $N$.
Hence, equation (\ref{eq:RienLoinDuReel}) implies that most of the resonances in $\mathcal{R}_{a_1,a_2, - a_3, 0}$ are actually in $\mathcal{R}_{a_1,a_2, - \delta_N, 0}$, provided we have $\delta_N \gg \left(\frac{\gamma_N}{N}\right)^{2/5}$. The quantity $\left(\frac{\gamma_N}{N}\right)^{2/5}$ may thus be thought of as a \textcolor{black}{lower bound for the imaginary part of typical resonances}.

We would like to insist that we don't believe the power $\frac{5}{2}$ appearing in (\ref{eq:RienLoinDuReel}) to be optimal. It would be very interesting to carry out numerical simulations to see what is the best power we can expect here.
\end{rem}

\begin{exemple}
Let $A_N$ be a Hermitian $N\times N$ matrix, and let $B_N$ be a negative Hermitian $N\times N$ matrix. Let $M_N := A_N + iB_N$. Suppose that $\|A_N\|, \|B_N\|\leq C$ for some $C$ independent of $N$.

We want to show that, if $\|B_N\|_1$ is small, $M_N$ has few eigenvalues away from the real axis.

The matrix $B_N$ being negative Hermitian, if $z$ is an eigenvalue of $M_N$, we must have $\Im z \leq 0$. \textcolor{black}{Furthermore, since $\|A_N\|, \|B_N\|\leq C$, if $z$ is an eigenvalue of $M_N$, we must have $|z|\leq 2C$.}

For all $z\in \C$ such that $\Im z <1$, the matrix $ \textcolor{black}{(z-i)} \mathrm{Id}- A_N - iB_N$ is invertible, and we set
$$U_N(z) := \big( \textcolor{black}{(z-i)} \mathrm{Id}- A_N - iB_N \big)^{-1}  \big( \textcolor{black}{(-z-i)} \mathrm{Id}+ A_N + iB_N \big).$$
We thus have
$$\left(U_N(z) v = v \right) \Longleftrightarrow \left( M_N v = z v\right).$$

Taking $Y_0 = \frac{1}{2}$, we have, for all $z\in \C$ with $\Im z < Y_0$, $ \left\|\left(\textcolor{black}{(z-i)} \mathrm{Id}- A_N - iB_N\right)^{-1}\right\| \leq 2$,  so that
$$\textcolor{black}{\|}U_N(z)\textcolor{black}{\|}\leq 2 (1+|z| + \|A_N\| + \|B_N\|).$$

\textcolor{black}{We have
\begin{align*}
U_N(z) - \mathrm{Id} = 2 \big( (z-i) \mathrm{Id}- A_N - iB_N \big)^{-1} \left( A_N + i B_N - z \mathrm{Id}\right).
\end{align*}}

\textcolor{black}{
In particular, when $z$ is not an eigenvalue of $M_N$, we have $\|\left(U_N(z) - \mathrm{Id}\right)^{-1}\|\leq  \frac{1+|z|+2C}{2} \left\| ( A_N +i B_N - z \mathrm{Id})^{-1}\right\|$. Points 2 and 3 of Hypothesis \ref{Hyp:GenRes} follow easily, taking $Y_1 = 3C$.}

We also define for all $z\in \C$ such that $\Im z <1$
$$\widetilde{U}_N(z) := \left( \textcolor{black}{(z-i)} \mathrm{Id}- A_N  \right)^{-1}  \left( \textcolor{black}{(-i-z)} \mathrm{Id}+ A_N  \right).$$

\textcolor{black}{Note that $\widetilde{U}_N(z)$ can be diagonalized in an orthonormal basis. More precisely , if $v\in \R^N$ and  $\lambda\in \R$ are such that $Av = \lambda v$, we have $ \widetilde{U}_N(z) v =  \left((z-i) \mathrm{Id}- \lambda  \right)^{-1}  \left((-i-z) \mathrm{Id}+ \lambda  \right)v$. Equation (\ref{eq:BorneSurLesNormes}) follows easily.}

Finally, we have
\begin{align*}
\left\| U_N(z) - \widetilde{U}_N(z)\right\|_1 &=\Big\| i   \big( \textcolor{black}{(z-i)} \mathrm{Id}- A_N - iB_N \big)^{-1}  B_N  \\
&+  \left[ \big( \textcolor{black}{(z-i)} \mathrm{Id}- A_N - iB_N \big)^{-1} - \left( \textcolor{black}{(z-i)} \mathrm{Id}- A_N  \right)^{-1}   \right]   \big( \textcolor{black}{(-i-z)} \mathrm{Id}+ A_N  \big) \Big\|_1\\
&\leq  2 \big(1+|z| + \|A_N\| + \|B_N\|\big) \|B_N\|_1 \\
& + 2 \big{(}1+|z| + \|A_N\|\big{)} \left\| \big( \textcolor{black}{(z-i)} \mathrm{Id}- A_N - iB_N \big)^{-1} \right\| \left\| \big( \textcolor{black}{(z-i)} \mathrm{Id}- A_N  \big)^{-1} \right\| \|B_N\|_1\\
&\leq C(z, \|A_N\|, \|B_N\|) \|B_N\|_1.
\end{align*}

Therefore, our matrices satisfy all the assumptions from Hypotheses \ref{Hyp:GenRes} and \ref{Hyp:TildeU}, with $\gamma_N = \|B_N\|_1$. We deduce that, for any sequence $(\delta_N)$ of positive numbers, the number of eigenvalues of $M_N$ in $\mathcal{R}_{a_1,a_2, - a_3, - \delta_N}$ is a $O\left( \frac{\|B_N\|_1}{\delta_N^{5/2}}\right)$.

This result seems elementary, but we could not find a trace of it in the literature.
\end{exemple}

To state our next theorem, we will need to restrict ourselves to the finite-dimensional setting. The main reason for this is that we want the operators $\widetilde{U}_N$ to be unitary on the real axis; this assumption is compatible to being trace-class only in finite dimension.

\begin{tcolorbox}
\begin{hyp}\label{HypRenforcee}
\begin{enumerate}
\item There exists $D_0, D_1$ such that $\mathcal{H}_N = \C^{d_N}$, with
\begin{equation}\label{eq:BornesDimension}
D_0 N \leq d_N \leq D_1 N.
\end{equation}

\item The sequence $(\gamma_N)$ from Hypothesis \ref{Hyp:TildeU} satisfies $\gamma_N = o(N)$.

\item The operators  $\widetilde{U}_N(x)$ from Hypothesis \ref{Hyp:TildeU} are unitary for all $x\in \R$.

\item  For any compact set $\mathcal{K}\subset \R$, there exists $c(\mathcal{K})>0$ such that for all $x\in \mathcal{K}$ and all $N\in \N$,
\begin{equation}\label{eq:Positivite}
-i \frac{\mathrm{d}\widetilde{U}_N(x) }{\mathrm{d}x} \widetilde{U}_N(x)  - c(\mathcal{K}) \mathrm{Id} ~~ \text{ is a positive self-adjoint operator}.
\end{equation}

\item For any $z\in \C$, the matrix $U_N(z)$ is \textcolor{black}{similar} to $U_N^*(-\overline{z})$ 
, and the matrix $\widetilde{U}_N(z)$ is \textcolor{black}{similar} to $\widetilde{U}^*_N(-\overline{z})$.
\end{enumerate}
\end{hyp}
\end{tcolorbox}

Note that, thanks to (\ref{eq:BornesDimension}), if $U : \mathcal{H}_N \longrightarrow \mathcal{H}_N$, we have
\begin{equation}\label{eq:TraceBound}
\|A\|_1 \leq D_1 N \|A\|.
\end{equation} 

Furthermore, the last assumption of Hypothesis \ref{HypRenforcee} implies that the resonances of $U_N$ and $\widetilde{U}_N$ are symmetric with respect to the imaginary axis. 

Our next theorem says that, in boxes close to the real axis, $U_N$ and $\widetilde{U}_N$ have asymptotically the same number of resonances, up to changing slightly the size of the boxes.

\begin{tcolorbox}[breakable, enhanced]
\begin{theoreme}\label{th:ResClose}
Let $(U_N)$ and $(\widetilde{U}_N)$ satisfy Hypotheses \ref{Hyp:GenRes}, \ref{Hyp:TildeU} and \ref{HypRenforcee}.

Let $a<b$, and let $c>0$. 
For any $\varepsilon>0$, we have
\begin{align*}
 \limsup\limits_{N\to +\infty} \frac{1}{N}\mathcal{N}_{a+ \varepsilon, b- \varepsilon, -ic, 0}(U_N) &\leq \limsup\limits_{N\to +\infty} \frac{1}{N}\mathcal{N}_{a, b}(\widetilde{U}_N)\\
  \liminf\limits_{N\to +\infty} \frac{1}{N}\mathcal{N}_{a- \varepsilon, b+ \varepsilon, -ic, 0}(U_N)&\geq \liminf\limits_{N\to +\infty} \frac{1}{N}\mathcal{N}_{a, b}(\widetilde{U}_N)
 \end{align*}
\end{theoreme}
\end{tcolorbox}

Theorem \ref{th:ResClose} becomes more explicit if we consider the empirical spectral measures of $U_N$ and $\widetilde{U}_N$, and suppose that the empirical spectral measures of $\widetilde{U}_N$ converge. 

 For any $N\in \N$, let us denote by $(x_{j,N})_{j\in J_N}$ the resonances of $(U_N)$, and by $(x_{j,N})_{j\in \widetilde{J}_N}$ the resonances of $\widetilde{U}_N$.
For each $N$, these resonances can be in finite or infinite number, but are isolated. Therefore, the measure
$$\mu_N := \sum_{j\in J_N} \delta_{x_{j,N}}$$
is a locally finite measure on $\C$, while 
$$\widetilde{\mu}_N := \sum_{j\in \widetilde{J}_N} \delta_{x_{j,N}}$$
is a locally finite measure on $\R$.

\begin{tcolorbox}
\begin{corollaire}\label{cor}
Let $(U_N)$ and $(\widetilde{U}_N)$ satisfy Hypotheses \ref{Hyp:GenRes}, \ref{Hyp:TildeU} and \ref{HypRenforcee}. 
Suppose that there exists a locally finite Borel measure $\mu$ on $\R$ such that $\frac{1}{N}\widetilde{\mu}_N$ converges vaguely to $\mu$.

Then $\frac{1}{N}\mu_N$ converges vaguely to $\mu$.
\end{corollaire}
\end{tcolorbox}

\begin{proof}[Proof that Theorem \ref{th:ResClose} implies Corollary \ref{cor}]

\textbf{Step 1:}
Let us fix $c>0$. We define a sequence of measures $\nu_{N\textcolor{black}{,c}}$ on $\R$ by 
$$\nu_{N,\textcolor{black}{c}}(\mathcal{I}) =\frac{1}{N} \times \left( \text{number of resonances of $U_N$ in $\mathcal{I}+ i [-c, 0]$}\right).$$

Let $a'<b'$ and $0<\varepsilon < \frac{b'-a'}{6}$. We have
\begin{align*}
\liminf\limits_{N\to +\infty} \nu_{N,c}\big((a', b' )\big) &\geq  \liminf\limits_{N\to +\infty} \nu_{N,c}\big([a'+\varepsilon, b'- \varepsilon] \big)\\
&\geq  \liminf\limits_{N\to +\infty} \widetilde{\mu}_N\big([a'+2\varepsilon, b'- 2\varepsilon] \big)~~\text{by Theorem \ref{th:ResClose}}\\
&\geq \liminf\limits_{N\to +\infty} \widetilde{\mu}_N\big((a'+2\varepsilon, b'- 2\varepsilon) \big)\\
&\geq \mu \big((a'+2\varepsilon, b'- 2\varepsilon) \big) ~~\text{ by the Portmanteau theorem}\\
&\geq  \mu \big([a'+3\varepsilon, b'- 3\varepsilon] \big).
\end{align*}

Now, recall that the measure $\mu$ is regular, so that the limit of the right-hand side as $\varepsilon\to 0$ is $\mu\big( (a,b)\big)$. All in all, we have, for all $a'<b'$
$$\liminf\limits_{N\to +\infty} \nu_{N,c}\big((a', b' )\big) \geq  \mu\big((a',b')\big).$$
We may thus apply Portmanteau theorem to deduce that, for any $c>0$, $\nu_{N,c}$ converges vaguely to $\mu$ as $N\to +\infty$.

\textbf{Step 2:}
Let $f\in C_c^\infty(\C)$. Our aim is to show that $\mathbb{E}_{\mu_N} \left[f\right] \longrightarrow \mathbb{E}_{\mu} \left[f\right]$. 

For any $\delta>0$, let $\chi_\delta \in C_c^\infty(\R; [0,1])$ be supported in $(-2\delta,  2\delta)$, and equal to one on $(-\delta, \delta)$. We define two functions $\widehat{f}_\delta, \widehat{f}'_\delta \in C^\infty(\C)$  by
\begin{align*}
\widehat{f}_\delta (x+i y) &= \chi_\delta(y) f(x+iy)\\
\widehat{f}'_\delta (x+i y) &= \chi_\delta(y) f(x).
\end{align*}

Note that, when $f$ takes only non-negative values, we have for any $\delta>0$
$$\E_{\nu_{N,\delta}}[f]\leq \E_{\mu_N} \left[\widehat{f}'_\delta\right] \leq \E_{\nu_{N, 2\delta}}[f],$$
so that, by the first step, for any $\delta>0$,
\begin{equation}\label{eq:Convf}
\E_{\mu_N} \left[\widehat{f}'_\delta\right] \underset{N\to +\infty}{\longrightarrow} \E_{\mu}[f].
\end{equation}

Decomposing $f$ into its positive and negative parts shows that (\ref{eq:Convf}) actually holds for all $f\in C_c^\infty(\C)$.

On the other hand, by continuity, $\|\widehat{f}_\delta- \widehat{f}'_\delta \|_{C^0} = O(\varepsilon)$. Therefore, using Theorem \ref{th:FewRes2} and our assumption, we have
\begin{equation}\label{eq:limsup1}
\limsup\limits_{N\to +\infty} \E_{\mu_N}\left[ \left|\widehat{f}_\delta- \widehat{f}'_\delta\right|\right] =O(\delta).
\end{equation}

Finally, for any $\delta>0$, the function $f- \widehat{f}_\delta$ is supported away from the real axis, so, by Theorem \ref{th:FewRes2}, we have
\begin{equation}\label{eq:limsup2}
\limsup\limits_{N\to +\infty} \E_{\mu_N}\left[ \left| f- \widehat{f}_\delta\right|\right] =0.
\end{equation}

Combining equations (\ref{eq:Convf}), (\ref{eq:limsup1}) and (\ref{eq:limsup2}) gives us the result.
\end{proof}

\section{Open quantum graphs and their resonances}\label{sec:QG}
We will now recall the definition of resonances on open quantum graphs, and explain how they enter the framework developed in the previous section.

\subsection{\textcolor{black}{Definition of resonances on open quantum graphs}}
An (open) quantum graph is given by a finite graph, where each edge is given a length and where infinite edges (called \emph{leads}) are attached to some of the vertices. More precisely
\begin{tcolorbox}
\begin{definition}
A \emph{quantum graph} $\mathcal{Q}=(V,E,L, \mathbf{n})$ is the data of
\begin{itemize}
\item A graph $G=(V,E)$  with vertex set $V$ and edge set $E$.
\item A map $L: E\rightarrow (0,\infty)$.
\item A map $\mathbf{n}: V \longrightarrow \N\cup \{0\}$.
\end{itemize}

The quantum graph will be called \emph{finite} if $G$ is a finite graph. We will say that $\mathcal{Q}$ is a \emph{closed} quantum graph if $\boldsymbol{n}(v) = 0$ for all $v\in V$.
\end{definition}
\end{tcolorbox}

All the quantum graphs we will consider will always be either finite or closed.

The graph $(V,E)$ should be thought of as the compact part of our graph, the map $L$ gives the length of the edges in the compact part, while the map $\mathbf{n}$ gives the number of infinite edges attached to each vertex. If $v\in V$, we denote by $\textcolor{black}{d(v)}$ the (internal) degree of $v$, i.e., the number of $e\in E$ to which $v$ belongs. We define the total length of the graph by
$$\mathcal{L}(\mathcal{Q}):= \sum_{e\in E} L(e).$$

We let $B= B(G)$ be the set of oriented internal edges (or bonds) associated to $E$. If $b\in B$, we shall denote by $\hat{b}$ the reverse bond. We write $o(b)$ for the origin of $b$ and $t(b)$ for the terminus of $b$. 

We define the external bonds by $B_{ext} := \bigsqcup_{v\in V} \bigsqcup _{k=1}^{\mathbf{n}(v)}\textcolor{black}{ \{}(v, k) \textcolor{black}{\}}$, and we write $o_b= v$ if $b=(v,k)\in B_{ext}$, and set $\hat{B}= B\cup B_{ext}$.

A $C^2$ function on the graph will be a collection of maps $f=(f_b)_{b\in \hat{B}}$, with $f_b\in C^2([0,L_b])$ if $b\in B$, $f_b\in C^2([0,\infty))$ if $b\in B_{ext}$, and such that
\begin{equation} \label{eq:CondSym}
\forall b\in B, f_b (\cdot) = f_{\hat{b}}(L_b - \cdot).
\end{equation}
 We will write $C^2(\mathcal{Q})$ for the set of such functions. We may now define the boundary conditions we put at the vertices.

\begin{tcolorbox}
\begin{definition}\label{def:Kirch}
We say that $f\in C^2(\mathcal{Q})$ satisfies \emph{Kirchhoff boundary conditions} if it satisfies
\begin{itemize}
\item \textbf{Continuity}: For all $b,b'\in \hat{B}$, we have $f_b(0) = f_{b'}(0) =:f(v)$ if $o_b=o_{b'}=v$.\\
\item \textbf{Current conservation:} For all $v\in V$, 
\begin{equation*}
\sum_{b\in \hat{B}:o(b)=v} f_b'(0)= \textcolor{black}{0}\,.
\end{equation*}
\end{itemize}
\end{definition}
\end{tcolorbox}

Let us now move to the definition of scattering resonances of quantum graphs.

\begin{tcolorbox}
\begin{definition}\label{def:Res}
Let $\cQ$ be a finite quantum graph.
A number $z\in \C$ is called a \emph{scattering resonance} of $\cQ$ if there exists $f\in C^2(\cQ)$ such that
\begin{enumerate}
\item $f$ satisfies the Kirchhoff boundary conditions.
\item For all $b\in \hat{B}$, we have $-f_b''= z^2 f_b$.
\item For all $b\in B_{ext}$, we have $f_b(x) = f_b(0) e^{i zx}$.
\end{enumerate}

If $z$ is a resonance, we define its multiplicity, written $\mathrm{mult}(z)$ as the number of linearly independent $f\in C^2(\cQ)$ satisfying 1., 2. and 3.

We will write
$$\mathrm{Res}(\mathcal{Q}):= \{ \text{Resonances of } \mathcal{Q}\} \subset \C.$$

Finally, if $\Omega\subset \C$, we write
$$\mathcal{N}_{\Omega}(\cQ) := \sum_{z \in \Omega \cap \mathrm{Res}(\mathcal{Q})} \mathrm{mult}(z).$$
\end{definition}
\end{tcolorbox}

\begin{rem}
If the quantum graph $\mathcal{Q}$ is closed, then condition 3. above is empty. A scattering resonance is then just the square-root of an eigenvalue of the Laplacian on $\cQ$ with suitable selfadjoint conditions. We refer the reader to \cite[Chapter 1 and 2]{BK} for more details on this situation. In particular, all the scattering resonances of a closed quantum graph belong to the real axis.
\end{rem}

\begin{rem}\label{rem:Origin}
When $z=0$, condition 3. means that $f_b$ is constant on every lead. In particular, the multiplicity of $z=0$ does not change when we add or remove leads, i.e., if we change the value of $\mathbf{n}$.
\end{rem}

Note that, if we did not impose the last condition (and if $\mathbf{n}$ is not identically zero), then any number $z\in \C$ would be a scattering resonance. On the other hand condition 3. imposes that scattering resonances must have negative imaginary part. Otherwise, the function $f$ would be an eigenfunction of a selfadjoint operator, associated to a non-real eigenvalue. We will see another proof of this fact in the next subsection.

\subsection{A characterisation of scattering resonances of an open quantum graph}
If a function $f\in C^2(\mathcal{Q})$ satisfies 1. and 3. in Definition \ref{def:Res}, then for every $b\in B_{ext}$, $f_b$ is entirely determined by $f_b(0)$, and we have $f'_b(0)= iz f_b(0)$, so that the current conservation condition rewrites
\begin{equation}\label{eq:GenKirch}
\forall v\in V, \sum_{b\in B ; o_b=v} f'_b(0) = -i \mathbf{n}(v) z  f(v).
\end{equation}

\begin{rem}
In particular, we see that $z$ is a resonance if and only if there exists $\tilde{f}= (f_b)_{b\in B}$ where $f_b\in C^2([0,L_b])$ satisfies (\ref{eq:CondSym}) and $-f_b''=z^2f_b$, and such that the continuity condition and (\ref{eq:GenKirch}) are satisfied.

This caracterization has the advantage of only implying (functions on) the compact part of the graph, along with the map $\mathbf{n}$. 
\end{rem}

The condition $-f_b''=z^2f_b$ implies that, for each $b\in \hat{B}$, we can find coefficients $a(b), a'(b)\in \C$ such that\footnote{The third condition in Definition \ref{def:Res} can hence be thought of as imposing $a'(b)=0$ if $b\in B_{ext}$: the wave is made of purely \emph{outgoing} parts on the external bonds.} $f_b(x)= a(b) e^{izx}+ a'(b) e^{-izx}$.

Let us rewrite the generalized Kirchhoff conditions in terms of the coefficients $a(b)$ and $a'(b)$.
The continuity condition rewrites
$$\forall b,b'\in B \text{ such that } o_b=o_{b'}= v, \text{ we have } a(b)+a'(b)= a(b')+ a'(b') = f(v).$$
The condition about derivatives can be expressed as
\begin{align*}
   \boldsymbol{n}(v) f(v) &= - \sum_{b\in B ; o_b =v} \left[a(b) - a'(b)\right]\\
&= -  d(v) f(v) + 2\sum_{b\in B ; o_b =v} a'(b),
\end{align*}
which amounts to
$$\forall b\in B,~  a(b) = \sum_{b'\in B} \sigma_{b,b'} a'(b'),$$
with
$$\sigma_{b,b'} = \begin{cases} \frac{2}{\boldsymbol{n}(v)+d(v)} &\text{ if } o_b=o_{b'}=v \text{ and } b'\neq b\\
\frac{2}{\boldsymbol{n}(v)+d(v)} - 1 &\text{ if } b'=b \text{ with } o_b=v\\
0 &\text{ if } o_b\neq o_{b'}. \end{cases}$$

\begin{rem}\label{rem:NormeSigma}
Note that if we denote by $\sigma^{(v)}$ the $d(v)\times d(v)$ matrix whose lines and columns are indexed by the bonds of $B$ such that $o_b=v$, with entries $\sigma^{(v)}_{b,b'} = \sigma_{b,b'}$, then 
$$\sigma^{(v)}= - \mathrm{Id} + \frac{2}{\boldsymbol{n}(v) + d(v)}  J,$$
where $J$ has all its entries equal to one. In particular, $\sigma^{(v)}$ is symmetric and has eigenvalue $-1$ with multiplicity $d(v)-1$, and $ -1 + \frac{2 d(v)}{\boldsymbol{n}(v)+d(v)}$ with multiplicity one. In particular, $\sigma^{(v)}$ is always subunitary.

When $\boldsymbol{n}(v)=0$ (hence, when there are no leads attached to $v$), then $\sigma^{(v)}$ is an isometry. Otherwise, we see that $\sigma^{(v)}$ is invertible unless $v$ is a \emph{balanced} vertex, i.e., a vertex such that $\mathbf{n}(v)= d(v)$.
\end{rem}

Finally, recall that, condition (\ref{eq:CondSym}) implies that for all $b\in B$, we have $a'(b) = e^{iz L_b} a(\hat{b})$, so that
$$f_b(x) = a(b) e^{izx} + a(\hat{b}) e^{i z (L_b-x)}.$$

The continuity and generalized Kirchhoff conditions may then be rewritten as
\begin{equation}\label{eq:CondCoef}
a(b) = \sum_{b'\in B} \sigma_{b,\hat{b'}} e^{i z L_{b'}} a(b').
\end{equation}

Let us define the matrices $D(z)$, $S$ and $U(z)$ whose lines and columns are indexed by the elements of $B$ by
\begin{equation}\label{eq:DefU}
\begin{aligned}
D(z)_{b,b'}&= \delta_{b,b'} \ee^{iz L_b}\\
S_{b,b'} &= \sigma_{b,\hat{b'}}\\
U(z) &= S D(z).
\end{aligned}
\end{equation}

An important feature of Kirchhoff boundary conditions, essential in our proof, and which disappears for more general conditions, is that the matrix $S$ does not depend on $z$.

Writing $\vec{a} = (a(b))_{b\in B}$, equation (\ref{eq:CondCoef}) then rewrites as
\begin{equation*}
U(z) \vec{a} = \vec{a}.
\end{equation*}

We have thus shown that, for any $z\in \C\backslash \{0\}$,
\begin{equation}\label{eq:CritRes}
z \text{ is a resonance } \Longleftrightarrow \mathrm{\det} (\mathrm{Id} - U(z)) = 0,
\end{equation}
and the multiplicities coincide.

This characterisation of resonances, which can be found for instance in \cite{KotSmi}, is reminiscent of the \emph{secular equation} introduced by Kottos and Smilansky in \cite{KS97}, which is a powerful tool to study the spectrum of closed quantum graphs.

\subsection{Families of open quantum graphs}
From now on, we consider a sequence of finite open quantum graphs $\mathcal{Q}_N=(V_N,E_N,L_N,\mathbf{n}_N)$, and we will be interested in the localization of its resonances. We shall always make the following assumption on our graphs.

\begin{tcolorbox}
\begin{hyp} [\textbf{Bounds}]\label{Hyp:Bounds}
There exists $D, n_0\in \N$, $0<L_{min}\leq L_{max}$ such that, for all $N\in \N$, we have
\begin{equation}\label{eq:Bounds}
\begin{cases}
& \forall v\in V_N, \textcolor{black}{d(v)}\leq D\\
&\forall e\in E_N, L_{min} \le L(e)\le L_{max}\\
& \forall v\in V_N, \mathbf{n}_N(v)  \le n_0.
\end{cases}
\end{equation}
\end{hyp}
\end{tcolorbox}

\textcolor{black}{We will always make the following assumption on our graphs.}

\begin{tcolorbox}
\begin{hyp}[\textbf{Unbalanced}]\label{Hyp:Unbalanced}
For all $N\in \N$, we have
\begin{equation}\label{eq:Unbalanced}
\forall v\in V_N, \mathbf{n}_N(v) \neq \textcolor{black}{d(v)}.
\end{equation}
\end{hyp}
\end{tcolorbox}

\begin{tcolorbox}
\begin{lemme}\label{lem:InABand}
Suppose that $(\mathcal{Q}_N)$ satisfies  \emph{\textbf{(Bounds)}} and \emph{\textbf{(Unbalanced)}}. Then 
\begin{equation}\label{eq:Band2}
 \mathrm{Res}(\cQ_N)\subset \R + i \left[0, \frac{1}{L_{min}} \sup_{v\in V_N} \big( \ln(\mathbf{n}(v)+ d(v)) - \ln|d(v) - \mathbf{n}(v)|\big)\right].
 \end{equation}
 
In particular,
\begin{equation*}
\mathrm{Res}(\cQ_N)\subset \R + i \left[- \frac{\ln(D+n_0)}{L_{min}}, 0\right].
\end{equation*}
\end{lemme}
\end{tcolorbox}

\begin{proof}

Under hypothesis \textbf{(Unbalanced)}, we see that the matrices $\sigma^{(v)}$ are all invertible, and we have $\|(\sigma^{(v)})^{-1}\| = \frac{\mathbf{n}(v)+ d(v)}{|d(v) - \mathbf{n}(v)|} \leq \mathbf{n}(v)+ d(v).$
In particular, if we furthemore suppose that (\textbf{Bound}) is satisfied, we get
$$\|\sigma^{(v)}\|^{-1}\leq n_0 + D.$$

We see from Remark \ref{rem:NormeSigma} that $\|S\|\leq 1$.
Furthermore, if $J$ is the $B\times B$ matrix such that $J_{b,b'} = \delta_{b', \widehat{b}}$, then 
the matrix $SJ$ is a block matrix, so it can be inverted block by block. We deduce from this and from Remark \ref{rem:NormeSigma} that $\|S^{-1}\|\leq n_0 + D$. In particular, we see that $U(z)$ is invertible, and that for any $z\in \C^-$, we have
$$\|U(z)^{-1}\| \leq (n_0 + D) e^{\Im z L_{min}}.$$

Recalling that $\mathrm{\det} (\mathrm{Id} - U(z)) = 0$ if and only if $\mathrm{\det} (\mathrm{Id} - U^{-1}(z)) = 0$, we deduce (\ref{eq:Band2}).
\end{proof}

\subsubsection{Benjamini-Schramm convergence for open quantum graphs}\label{sec:BS}
We do now introduce the definition of Benjamini-Schramm convergence of open quantum graphs, following closely \cite[\S 3]{BSQG} where similar definitions are introduced for closed quantum graphs.

A rooted open quantum graph $(\mathcal{Q},b_0)=(V,E,L,\mathbf{n},b_0)$ will be the data of an open quantum graph $\mathcal{Q}=(V,E,L,\mathbf{n})$, and of a bond $b_0\in B(G)$.
\begin{tcolorbox}
\begin{definition}
We say that two rooted open quantum graph $(\mathcal{Q}_0,b_0)= (V_0,E_0,L_0,\mathbf{n}_0,b_0)$ and $(\mathcal{Q}_1,b_1)=(V_1,E_1,L_1,\mathbf{n}_1,b_1)$ are equivalent, which we denote by $(\mathcal{Q}_0,b_0)\sim (\mathcal{Q}_1,b_1)$, if there exists a graph isomorphism $\phi : (V_0,E_0)\longrightarrow (V_1,E_1)$ such that $\phi(o(b_0))= o(b_1)$, $\phi(t(b_0))= t(b_1)$, $L_1\circ \phi = L_0$, and $\mathbf{n}_1 \circ \phi = \mathbf{n}_0$.

The set of rooted open quantum graphs, quotiented by $\sim$, will be denoted by $\mathrm{ROQ}$. If $(\cQ, b_0)$ is a rooted open quantum graph, we will denote by $[\cQ,b_0]$ its equivalence class.
\end{definition}
\end{tcolorbox}

If $v\in G$ is a vertex in a graph and $r\in \N$, we write $\mathrm{B}_G(v,r)$ for the set of vertices which are at a (combinatorial) distance at most $r$ from $v$. We write $E(\mathrm{B}_G(v,r))$ for the set of edges in $E$ connecting two vertices of $\mathrm{B}_G(v,r)$.

We introduce a distance between rooted quantum graphs as follows
\begin{align*}
\mathrm{d}\left([\cQ_1,b_1], [\cQ_2,b_2]\right) := \inf \Big{\{} &\varepsilon>0 ~ \big{|} ~  \exists \phi : \mathrm{B}_{G_1}(o_{b_1}, \lfloor\varepsilon^{-1}\rfloor) \to  \mathrm{B}_{G_2}(o_{b_2}, \lfloor\varepsilon^{-1}\rfloor) \text{ a graph isomorphism }\\
& \text{ such that } \mathbf{n}_2 \circ \phi = \mathbf{n}_1 \text{ and } \sup\limits_{e \in E(\mathrm{B}_{G_1}(o_{b_1}, \lfloor\varepsilon^{-1}\rfloor))} |L_2(\phi(b)) - L_1(b)|<\varepsilon \Big{\}}.
\end{align*}

Note that this definition is independent of the representatives we chose in the equivalence classes $[\cQ_1,b_1], [\cQ_2,b_2]$, so it is well-defined on $\mathrm{ROQ}$. Furthermore, one can show that $(\mathrm{ROQ}, d)$ is a Polish space, i.e., a separable complete metric space.

Let $\cP(\mathrm{ROQ})$ be the set of Borel probability measures on $\mathrm{ROQ}$.

\begin{tcolorbox}
\begin{definition}
Any finite quantum graph $\mathcal{Q}=(V,E,L,\mathbf{n})$ defines a probability measure $\nu_{\cQ}\in \cP(\mathrm{ROQ})$ obtained by choosing a root uniformly at random: 
\[
\nu_{\cQ}:= \frac{1}{|B|} \sum_{b_0\in B}  \delta_{[(\cQ,b_0)]}.
\]

If $\textcolor{black}{(\cQ_N)}$ is a sequence of quantum graphs, we say that $\mathbb{P}\in \cP(\mathrm{ROQ})$ is the \emph{local weak limit} of $\textcolor{black}{(\cQ_N)}$, or that $\textcolor{black}{(\cQ_N)}$ \emph{converges in the sense of Benjamini-Schramm to $\mathbb{P}$}, if $\textcolor{black}{(\nu_{\cQ_N})}$ converges weakly to $\mathbb{P}$.
\end{definition}
\end{tcolorbox}

Note that, in the special case of equilateral graphs with no leads, that is, if, for all $N$, we have $\textbf{n}_N \equiv 0$ and $L_N \equiv L_0$, then $\textcolor{black}{(\cQ_N)}$ converges in the sense of Benjamini-Schramm if and only if the underlying discrete graphs $(V_N,E_N)$ converge in the sense of Benjamini-Schramm (recall that Benjamini-Schramm convergence for discrete graphs was discribed in section \ref{sec1.3}). We refer the reader to \cite[\S8]{BSQG} for examples of sequences of closed quantum graphs that converge in the sense of Benjamini-Schramm.

Let $D, n_0 \in \N$, $0<m\leq M$. We define $\mathrm{ROQ}^{D, n_0, L_{min}, L_{max}}\subset \mathrm{ROQ}$ as the subset of equivalence classes $[\cQ,b_0]=[(V,E,L,\mathbf{n}, b_0)]$ 
such that $\cQ$ satisfies (\ref{eq:Bounds}).

\begin{tcolorbox}
\begin{lemme}\label{lem:QCompact}
The subset $\mathrm{ROQ}^{D, n_0,  L_{min}, L_{max}}$ is compact.

 In particular, using Prokhorov's theorem, we see that if $(\mathcal{Q}_N)$ is a sequence of finite open quantum graphs which satisfy \textbf{(Bounds)}, then  there is a subsequence $(\cQ_{N_k})$ which converges in the sense of Benjamini-Schramm (i.e. there exists $\mathbb{P}\in \mathcal{P}(\mathrm{ROQ})$ supported on $\mathrm{ROQ}^{D, n_0, L_{min}, L_{max}}$ such that $\nu_{{Q}_{N_k}}\xrightarrow{w^*} \mathbb{P}$).
\end{lemme}
\end{tcolorbox}

Let us now present the main result of \cite{BSQG}, showing that Benjamini-Schramm convergence of closed quantum graphs implies the convergence of their empirical spectral measures.

If $\mathcal{Q}$ is a finite quantum graph (open or closed), we define its empirical spectral measure as
$$\mu_\mathcal{Q}:= \sum_{z\in \mathrm{Res}(\cQ)} \delta_{z}.$$

\begin{tcolorbox}
\begin{theoreme}\label{th:Rappel}
Let $(\cQ_N)$ be a sequence of closed quantum graphs satisfying Hypothesis \textbf{(Bounds)}. Then $\textcolor{black}{\left(\frac{1}{\mathcal{L}(\cQ_N)}\mu_{\cQ_N})\right)}$ converges vaguely to some measure  on $\R$, defined as follows
\begin{equation}\label{eq:BSQG2}
\frac{1}{\mathcal{L}(\cQ_N)} \sum_{x\in \mathrm{Res}(\mathcal{Q}_N)} \chi(x) \underset{N\to +\infty}{\longrightarrow} \int_\R \chi(\sqrt{x}) \mathrm{d}\mu_{\mathbb{P}}(x) =: \int_{\mathbf{Q}_*} \chi\big{(}\sqrt{H_\cQ}\big{)}(\mathbf{x_0},\mathbf{x_0}) \dd \mathbb{P}(\cQ,\mathbf{x_0}),
\end{equation}
where $H_\cQ$ is the Schr\"odinger operator on the limiting random quantum graph $\cQ$, $\chi\big{(}\sqrt{H_\cQ}\big{)}(\mathbf{x_0},\mathbf{x_0})$ is the value of the Schwartz kernel of the operator $\chi\big{(}\sqrt{H_\cQ}\big{)}$ at the root $\mathbf{x_0}$ and $\mathbf{Q}_*$ is a set of (equivalence classes of) rooted
quantum graphs.
\end{theoreme}
\end{tcolorbox}

We refer the reader to \cite{BSQG} for the fact that the operator $\chi\big{(}\sqrt{H_\cQ}\big{)}$ does indeed admit a continuous Schwartz kernel, and for other equivalent definitions of the measure $\mu_{\mathbb{P}}$.

\subsection{Resonances of sequences of open quantum graphs}

Just as in the previous section, if $a<b$, $c<d$, we will denote by  $\mathcal{N}_{a,b,c,d}(\mathcal{Q}_N)$ the number of resonances of $\mathcal{Q}_N$ with $ a\leq \Re z \leq b$ and  $c\leq \Im z \leq d$, and, more generally, by $\mathcal{N}_\Omega(\mathcal{Q}_N)$ the number of resonances of $\mathcal{Q}_N$ in $\Omega$.

Our main assumption is that the number of leads of $\mathcal{Q}_N$ is small compared to the number of vertices.

\begin{tcolorbox}
\begin{hyp}[\textbf{Weakly open}]
\begin{equation}
g(N):= \sharp \{ v\in V_N | \mathbf{n}_N(v) >0\} = o_{N\to +\infty} (|V_N|).
 \end{equation}
\end{hyp}
\end{tcolorbox}

We may then state the analogue of Theorem \ref{th:FewRes2}.

\begin{tcolorbox}[breakable, enhanced]
\begin{theoreme}\label{th:FewRes}
Let $(\mathcal{Q}_N)$ be a sequence of quantum graphs satisfying Hypotheses \emph{\textbf{(Bounds)}}, \textcolor{black}{\textbf{\emph{(Unbalanced)}}}  and \emph{\textbf{(Weakly open)}}. There exists a $C_0$ depending on the constants appearing in \emph{\textbf{(Bounds)}} such that the following holds.

Let $(\delta_N)$ be a bounded sequence of positive numbers, and let $0<a_1<a_2$, $0<a_3$. Then there exists $N_0$ such that for all $N\geq N_0$,
$$\mathcal{N}_{a_1,a_2, - a_3, - \delta_N} (\mathcal{Q}_N)\leq \frac{C g(N)}{\delta_N^{5/2}}.$$

In particular, if $\Omega\subset \C$ is a bounded set such that $\overline{\Omega}\cap \R = \emptyset$, we have
$$ \frac{1}{|V_N|}\mathcal{N}_{\Omega}(\cQ_N)\underset{N\to \infty}{\longrightarrow} 0.$$
\end{theoreme}
\end{tcolorbox}

\begin{proof}
We want to show how this enters the framework of Theorem \ref{th:FewRes2}.

Recall that the resonances of $\mathcal{Q}_N$ that are not zero are characterized, along with their multiplicity, by the equation 
$$\det (\mathrm{Id}- U_N(z)) = 0,$$
 with $U_N(z)$ as in (\ref{eq:DefU}). We must thus check that $U_N$ satisfies Hypotheses \ref{Hyp:GenRes} and \ref{Hyp:TildeU}. We may take $Y_0= +\infty$, as $U_N$ is well-defined on all of $\C$.

\textcolor{black}{When $\Im z>0$, we have $\|U_N(z)\| \leq \|S_N\| \|e^{izL_N}\| \leq e^{-\Im z L_{max}}$, and the first point of Hypothesis \ref{Hyp:GenRes} follows. In particular, if $\Im z>0$, then $\|U_N(z)\| <1$, so that $(\mathrm{Id} - U_N(z))$ can be inverted by a Neumann series, with $\left\|(\mathrm{Id} - U_N(z))^{-1}\right\| \leq \frac{1}{1- \|U_N(z)\|}\leq \frac{1}{1- e^{-\Im z L_{max}}}$. The second point of Hypothesis \ref{Hyp:GenRes} follows.}

\textcolor{black}{The third point of Hypothesis \ref{Hyp:GenRes} follows from 
Lemma \ref{lem:InABand} and its proof, which imply that all resonances lie in a band below the real axis whose size is bounded independently of $N$, and that outside of this band, the norms of the resolvents $\|U_N(z)- \mathrm{Id})^{-1}\|$ are bounded independently of $N$.}

Finally, the last point in Hypothesis \ref{Hyp:GenRes} follows from the definition of $U_N$ and Hypothesis \textbf{(Bounds)}. Indeed, we have $U_N(z) = S_N D_N(z)$, with $\|S_N\|\leq 1$, and, if $z$ belongs to some compact set $\mathcal{K}$, $\|D_N(z)\|\leq  \exp\left( L_{max} \min\limits_{z\in \mathcal{K}} \Im z\right)$.

Let us write $\widetilde{\mathcal{Q}}_N :=(V,E,L,0)$. In other words, $\widetilde{\mathcal{Q}}_N$ is the closed quantum graph obtained from removing all the leads from $\mathcal{Q}_N$. Let us denote by $\widetilde{U}_N(z) = \widetilde{S}_N D_N(z)$ the family of matrices defined as in (\ref{eq:DefU}), so that the non-zero resonances of $\widetilde{\mathcal{Q}}_N$ are characterized, along with their multiplicities, by the equation
$\det (\mathrm{Id}-\widetilde{U}_N(z))=0$.

The first point of Hypothesis \ref{Hyp:TildeU}, when $\Im z>0$, is proved in exactly the same way as the first point of Hypothesis \ref{Hyp:GenRes}. When $\Im z<0$, we note that, from Remark \ref{rem:NormeSigma}, $\widetilde{S}_N$ is an isometry, so that $\widetilde{U}_N(z)$ is invertible, and $\left\|\widetilde{U}_N(z)^{-1}\right\| \leq  e^{- L_{min}|\Im z|}$.
Hence, 
\begin{equation*}
\begin{aligned}
\left\|\left(\mathrm{Id}-\widetilde{U}_N(z)\right)^{-1}\right\| &\leq \|\widetilde{U}_N(z)^{-1}\|  \left\|\left( \mathrm{Id}-\widetilde{U}_N(z)\right)\right\|\\&\leq  e^{- L_{min}|\Im z|}  \left|1- e^{- L_{min}\Im z}\right|^{-1},
\end{aligned}
\end{equation*}
and the first point of Hypothesis \ref{Hyp:TildeU} follows.

For the second point of Hypothesis \ref{Hyp:TildeU}, note that
$$\left\|U_N(z) - \widetilde{U}_N(z)\right\|_1 \leq \|D_N(z)\| \|S_N - \widetilde{S}_N\|_1.$$

Now, $S_N - \widetilde{S}_N$ has all its entries bounded uniformly, and has a number of non-zero entries which is a $O(g(N))$. Therefore, the second point of Hypothesis \ref{Hyp:TildeU} follows, with $\gamma_N = g(N)$.

We may thus apply Theorem \ref{th:FewRes2} to deduce the result.
\end{proof}

Recall that, by Hypothesis \textbf{(Bounds)} and Lemma \ref{lem:QCompact}, we may always extract a subsequence from $(\mathcal{Q}_N)$ which converges in the sense of Benjamini-Schramm. Furthermore, by Hypothesis \textbf{(Weakly open)}, the Benjamini-Schramm limits of our graphs are always measures supported on $\mathrm{RQ}$, the set of equivalence classes of rooted quantum graphs such that $\boldsymbol{n}\equiv 0$ (in other words, the set of equivalence classes of closed quantum graphs). Therefore, the following hypothesis is always satisfied, up to extracting a subsequence; furthermore, it implies Hypothesis \textbf{(Weakly Open)}.

\begin{tcolorbox}
\begin{hyp}[\textbf{BS}]
The sequence $(\mathcal{Q}_N)$ converges in the sense of Benjamini-Schramm to a measure $\mathbb{P}\in \mathcal{P} (\mathrm{RQ})$.
\end{hyp}
\end{tcolorbox}

Our main result is then an analogue of Corollary \ref{cor}.

\begin{tcolorbox}[breakable, enhanced]
\begin{theoreme}\label{th:BS}
Let $(\mathcal{Q}_N)$ be a sequence of quantum graphs satisfying Hypotheses \emph{\textbf{(Bounds)}}, \textcolor{black}{\textbf{\emph{(Unbalanced)}}}   and \emph{\textbf{(BS)}}. Then $\mathcal{Q}_N$ satisfies (\ref{eq:BSQG2}). In particular, $\textcolor{black}{\left(\frac{1}{\mathcal{L}(\cQ_N)}\mu_{\mathcal{Q}_N}\right)}$ converges vaguely to some measure on $\R$ depending only on $\mathbb{P}$.
\end{theoreme}
\end{tcolorbox}
\begin{proof}
We want to apply Corollary \ref{cor}, with the notations of the proof of Theorem \ref{th:FewRes}. Note that the resonances of $U_N$ (resp $\widetilde{U}_N$) only correspond to the resonances of $\mathcal{Q}_N$ (resp $\widetilde{Q}_N)$ that are different from zero. This is not a problem since, by Remark \ref{rem:Origin}, zero has the same multiplicity as a resonance for $\mathcal{Q}_N$ and $\widetilde{\mathcal{Q}}_N$.

We need to check that Hypothesis \ref{HypRenforcee} holds, with the notations of the proof of Theorem \ref{th:FewRes}.

The first point always holds if we label our sequence of graphs such that $|V_N|=N$, and the second point then follows from Hypothesis \textbf{(Weakly open)}.

The third point follows from the fact that $\widetilde{S}_N$ is an isometry. Concerning the \textcolor{black}{fourth} point, note that 
\begin{align*}
 \frac{\mathrm{d}\widetilde{U}_N(x)}{\mathrm{d}x} \widetilde{U}_N^{-1}(x) &= \widetilde{S}_N iL_{\textcolor{black}{N}} e^{iL_Nx} e^{-iL_Nx} \widetilde{S}_N^{-1} = i\widetilde{S}_N L_N \widetilde{S}_N^*,
 \end{align*}
so that
$-i  \frac{\mathrm{d}\widetilde{U}_N(x)}{\mathrm{d}x} \widetilde{U}_N^{-1}(x) - L_{min} \mathrm{Id}$ is a positive selfadjoint operator, since $\widetilde{S}_N$ is an isometry and $\textcolor{black}{(L_N - L_{min})} \mathrm{Id}$ is selfadjoint and \textcolor{black}{non-negative}.

As for the last point, we have
\begin{align*}
U_N^*(-\overline{z}) &= \left(e^{-i \overline{z} L_N}\right)^* S_N^*\\
&= e^{iz L_N} S_N^*.
\end{align*}
\textcolor{black}{This is the transpose of the matrix $U_N(z)$, which is thus similar to $U_N(z)$.} The same proof shows that $\widetilde{U}_N(z)$ is \textcolor{black}{similar} to $\widetilde{U}^*_N(-\overline{z})$.

Therefore, Hypothesis \ref{HypRenforcee} is satisfied.

Finally, note that, since $(\cQ_N)$ converges in the sense of Benjamini-Schramm to $\mathbb{P}$, Hypothesis \textbf{(Weakly Open)} implies that $(\widetilde{\cQ}_N)$ does also converge to $\mathbb{P}$. Along with Theorem \ref{th:Rappel}, this shows that all the assumptions of Corollary \ref{cor} are fulfilled, which concludes the proof.

\end{proof}

\begin{rem}
\textcolor{black}{In what precedes, Hypothesis \emph{\textbf{(Unbalanced)}} is used only to obtain the bound (\ref{eq:HypResolv2}) from Hypothesis \ref{Hyp:GenRes}. We may thus wonder if this hypothesis is really necessary.}

\textcolor{black}{Actually, Theorem \ref{th:BS} cannot hold if a hypothesis like  \emph{\textbf{(Unbalanced)}} is not made, as can be seen in the following simple example.
Consider a graph with edges $x_0, ..., x_N$, with edges of length $L$ joining $x_j$ to $x_{j+1}$ for $0 \leq j \leq N-1$, and with a lead attached to $x_0$ and to $x_N$. Having a scattering resonance on such a graph is equivalent to having a scattering resonance for the free Laplacian on $\R$, which can happen only at $z=0$. However, if the leads are removed, resonances correspond to square-roots of the eigenvalues of the Neumann Laplacian on $[0, LN]$, which are regularly spaced at a distance $\frac{\pi}{LN}$, so that $\left(\frac{1}{N} \mu_{\widetilde{\mathcal{Q}}_N}\right)$ converges to a non-trivial measure. Hence, the result of Theorem \ref{th:BS} does not hold in this case.}
\end{rem}

\section{Tools from complex analysis}\label{sec:Complex}
The aim of this section is to recall the facts from complex analysis which we will use in the proof of Theorems \ref{th:FewRes2} and \ref{th:ResClose}.
\subsection{Jensen's formula}
Let $z_0\in \C$ and let $f$ be a holomorphic function such that $f(z_0)\neq 0$. For every $t>0$, denote by $n(t)$ the number of zeroes of $f(z)$ such that $|z-z_0|<t$. Recall that Jensen's formula tells us that
\begin{equation*}
\int_0^r \frac{n(t)}{t} \mathrm{d}t + \log |f(0)| = \frac{1}{2\pi} \int_0^{2\pi} \log |f(e^{i\theta} r)| \mathrm{d}\theta. 
\end{equation*}

Therefore, for any $\varepsilon>0$, we have

\begin{equation}\label{eq:Jensen10}
\begin{aligned}
n(r) &\leq \frac{1}{\log (1+ \varepsilon)}\int_r^{(1+\varepsilon) r} \frac{n(t)}{t} \mathrm{d}t\\
 &\leq \frac{1}{\log (1+ \varepsilon)} \left( \log \max_{|z-z_0|= (1+\varepsilon) r} |f(z)| - \log |f(z_0)| \right).
\end{aligned}
\end{equation}

\subsection{Consequences of the maximum principle}\label{sec:Max}
In all this section, we fix $z_0\in \C$ and $s,t>0$. We also take $0 <\alpha < 1$. We define the rectangle
$$\mathbf{R}_{z_0, s,t} := \mathcal{R}_{\Re z_0 - s, \Re z_0 + s, \Im z_0 - t, \Im z_0 + t},$$
and write
$$r_M:= \max(s,t), ~~ r_m:= \min(s,t).$$

Let $\mathcal{K}\subset \C$ be a bounded set containing  $\mathbf{R}_{z_0, s,t}$.


 Let $f$ be a holomorphic function on $\C$. We denote the zeroes of $f$ in $\mathcal{K}$, repeated with multiplicity, by $(z_j)_{j=1,..., J}$, with $J= J(f, \mathcal{K})$. Note that, if we take $\mathcal{K} = \mathbf{R}_{z_0, s,t}$, then, thanks to Jensen's formula, we have
 
 \begin{equation}\label{eq:JensenTrivial}
 J\leq \frac{1}{\log 2} \left( \max_{D(z_0, 2\sqrt{2}r_M)} \log |f| - \log |f(z_0)|\right).
 \end{equation}
 
We set
\begin{equation}\label{eq:DefBN}
\mathcal{B}_f(z):= \prod_{j=1}^{J} (z-z_j),
\end{equation}
so that, for all $z\in \mathcal{S}_{z_0, r}\subset \mathcal{K}$
\begin{equation}\label{eq:UpperB}
|\mathcal{B}_f(z)| \leq \mathrm{Diam}(\mathcal{K})^J.
\end{equation}

Since $\mathcal{B}_f$ has the same zeroes as $f$ in $\mathcal{K}$, we deduce that the map 
\begin{equation}\label{eq:DefGN}
 G_f :=  \frac{f}{\mathcal{B}_f}
 \end{equation}
is holomorphic on $\C$ and does not vanish in  $\mathcal{K}$.

We recall the following lemma, which was proven in \cite[Lemma 8.1]{Sj2}.
\begin{tcolorbox}
\begin{lemme}\label{lem:SjMagic2}
Let $x_1,..., x_J\in \R$, and let $I\subset \R$ be a bounded interval. Then there exists $x\in I$ such that
$$\prod_{j=1}^J |x-x_j| \geq \exp \left[-J \left(1 + \log \frac{2}{|I|}\right)\right].$$
\end{lemme}
\end{tcolorbox}

Note that for all $x, y\in \R$, we have
\begin{equation}\label{eq:LowerBlaschke3}
|\mathcal{B}_f(x +iy)| \geq \prod_{j=1}^{J} |x - \Re z_j|
\end{equation}
\begin{equation}\label{eq:LowerBlaschke4}
|\mathcal{B}_f(x +iy)| \geq \prod_{j=1}^{J} |y - \Im z_j|.
\end{equation}


We deduce from Lemma \ref{lem:SjMagic2} that we may find $x^1\in [\Re z_0 -s, \Re z_0-\alpha s]$, $x^2\in [\Re z_0+ \alpha s, \Re z_0 + s]$, $y^1\in [\Im z_0 - t, \Im z_0 - \alpha t]$, $y^2\in [\Im z_0 + \alpha t, \Im z_0 +t]$ such that, for $k=1,2$
\begin{equation}\label{eq:Confinement2.0}
\begin{aligned}
\forall y \in \R, ~~|\mathcal{B}_f(x^k +iy)| \geq \exp \left[-J \left(1 + \log \frac{2}{(1-\alpha) s}\right)\right]\\
\forall x\in \R,~~ |\mathcal{B}_f(x +iy^k)| \geq \exp \left[-J \left(1 + \log \frac{2}{(1-\alpha) t}\right)\right].
\end{aligned}
\end{equation}

We have just shown that
$$\forall z\in \partial \mathcal{R}_{x^1,x^2,y^1,y^2},~~ - \log |\mathcal{B}_f(z)| \leq J \left(1 + \log \frac{2}{\textcolor{black}{(1-\alpha)}r_m}\right).$$

Therefore, 
\begin{equation}\label{eq:AvantHarnack}
 \forall z\in \partial \mathcal{R}_{x^1,x^2,y^1,y^2},~~ |\log G_f(z)| \leq \log \sup_{z\in \mathbf{R}_{z_0, s,t}} |f(z)| + J \left(1 + \log \frac{2}{(1-\alpha) r_m}\right) =: M(f).
 \end{equation}

Now, since $\log |G_f(z)|$ is harmonic, the maximum principle tells us that (\ref{eq:AvantHarnack}) actually holds for all $z\in \mathbf{R}_{z_0, \alpha s ,\alpha t} \subset \mathcal{R}_{x^1,x^2,y^1,y^2}$.

Hence, the harmonic function $m(z) := 2M(f) - \log |G_f(z)|$ remains positive in $\mathbf{R}_{z_0, \alpha s ,\alpha t}$. We may thus apply Harnack's inequality to find a constant $C$ depending only on the ratio $\frac{s}{t}$ such that
\begin{equation}\label{eq:Harnack2}
\sup_{z\in \mathbf{R}_{z_0, \alpha s ,\alpha t} } m(z) \leq C \inf_{z\in \mathbf{R}_{z_0, \alpha s ,\alpha t}  } m(z) \leq C m(z_0).
\end{equation}

We deduce that, for any  $z\in \mathbf{R}_{z_0, \alpha s ,\alpha t}$, we have
\begin{equation*}
 |\log G_f(z)| \leq C \left(2 M(f) - \log |G_f(z_0)|\right)\leq C \big(2 M(f) + \left|\log |f(z_0)|\right| + J \left| \log \left( \mathrm{Diam}(\mathcal{K})\right) \right| \big). 
 \end{equation*}
 
 In other words, if we write
 \begin{equation}\label{eq:DefM}
 M(f,z_0,s,t,\alpha, \mathcal{K}) :=  \left(1+ \mathrm{Diam}(\mathcal{K}) +  |\log (1-\alpha) r_m| \right) \left(J(f, \mathcal{ K}) + \log \sup_{z\in \mathbf{R}_{z_0, s,t}} |f(z)| + \left|\log |f(z_0)| \right| \right),
 \end{equation}
we have, for all $z\in \mathbf{R}_{z_0, \alpha s ,\alpha t}$
 \begin{equation}
  |\log G_f(z)| \leq C  M(f,z_0,s,t,\alpha, \mathcal{K}).
 \end{equation}
 
  Since $\log |G_f|$ is harmonic, we deduce from standard elliptic gradient estimates (see for instance \cite[Theorem 3.9]{GilTru}) that
\begin{equation*}
\forall z \in \mathbf{R}_{z_0, \alpha^2 s,\alpha^2 t}, ~~ \left| \partial_z \log|G_f(z)|\right| \leq \frac{C}{r_m(1-\alpha)}  M(f,z_0,s,t,\alpha, \mathcal{K}). 
\end{equation*}

Now, $\log |G_f| = \Re \log G_f$, where $\log G_f$ is only well-defined modulo $2i\pi$, but is holomorphic. Hence, $\partial_z \log G_f(z)$ is well-defined, and by the Cauchy-Riemann equations, we have

\begin{equation}\label{eq:BorneDerivG}
\forall z \in \mathbf{R}_{z_0, \alpha^2 s,\alpha^2 t}, ~~ \left| \partial_z \log G_f(z) \right| \leq \frac{C}{r_m (1-\alpha)}  M(f,z_0,s,t,\alpha, \mathcal{K}). 
\end{equation}

We may sum up the results obtained in this section so far in the following lemma.

 \begin{tcolorbox}
\begin{lemme}\label{Lem:BorneDerivG}
Let $z_0\in \C$, let $s,t>0$, $0<\alpha<1$, and let $\mathcal{K}\subset \C$ be a bounded set containing $ \mathbf{R}_{z_0, s,t} := \mathcal{R}_{\Re z_0 - s, \Re z_0 + s, \Im z_0 - t, \Im z_0 + t}$.

There exists a constant $C= C(s/t)$ such that, if $f$ is a holomorphic function on $\mathbb{C}$, we have for all $ z \in \mathbf{R}_{z_0, \alpha^2 s,\alpha^2 t}$

\begin{equation*}
\left| \partial_z  \textcolor{black}{\log} \left(\frac{ f(z)}{\prod_{j=1}^J (z-z_j)}\right) \right| \leq \frac{C}{\min(s,t) (1-\alpha)}M(f,z_0,s,t,\alpha, \mathcal{K}),
\end{equation*}
with $M(f,z_0,s,t,\alpha, \mathcal{K})$ as in (\ref{eq:DefM}), and where $(z_j)_{j=1,...,J}$ are the zeroes of $f$ in $\mathcal{K}$.
\end{lemme} 
 \end{tcolorbox}
 
\paragraph{Lower bounds for $f$ on segments} 
 
 Let $I_1 \subset [\Re z_0 - \alpha s, \Re z_0 + \alpha s]$ and $I_2 \subset [\Im z_0 - \alpha t, \Im z_0 + \alpha t]$ be intervals.
Repeating the argument leading to (\ref{eq:Confinement2.0}), we see that we may find $x\in I_1$, $y \in I_2$ such that, for all $x'\in [\Re z_0 - \alpha s, \Re z_0 + \alpha s]$ and $y' \textcolor{black}{\in} [\Im z_0 - \alpha t, \Im z_0 +\alpha t]$, we have
\begin{equation}\label{eq:JeBoisTropDeCafe}
\begin{aligned}
| \mathcal{B}_f(x + i y')| &\geq \exp \left[ -J \left( 1 + \log \frac{2}{|I_1|}\right)\right]\\
| \mathcal{B}_f(x' + i y)| &\geq \exp \left[ -J \left( 1 + \log \frac{2}{|I_2|}\right)\right].
\end{aligned}
\end{equation}


Let us take $\mathcal{K}=\mathbf{R}_{z_0, s,t} $. Then, recalling (\ref{eq:DefM}) and (\ref{eq:JensenTrivial}), we deduce that there exists a constant $C=C(s,t)$ such that
 \begin{equation*}
 M(f,z_0,s,t,\alpha, \mathcal{K}) \leq  C \left(1 +  |\log (1-\alpha) r_m| \right) \left(\log \sup_{z\in D(z_0, 2\sqrt{2} r_M)} |f(z)|+ \left|\log |f(z_0)| \right| \right).
 \end{equation*}

Combining this with (\ref{eq:JeBoisTropDeCafe}), we obtain the following result.

\begin{tcolorbox}
\begin{lemme}\label{lem:GoodPoint}
Let $z_0\in \C$, let $s,t>0$, $0<\alpha<1$. Let $I_1 \subset [\Re z_0 - \alpha s, \Re z_0 + \alpha s]$ and $I_2 \subset [\Im z_0 - \alpha t, \Im z_0 + \alpha t]$ be intervals.
There exists a constant $C=C(s,t)$ such that the following holds.

For any $f$ holomorphic function on $\C$, we may find $x\in I_1$, $y \in I_2$ such that, for all $x'\in [\Re z_0 - \alpha s, \Re z_0 + \alpha s]$ and $y' \textcolor{black}{\in} [\Im z_0 - \alpha t, \Im z_0 +\alpha t]$, we have
\begin{equation*}
\begin{aligned}
\left| \log |f(x+iy')| \right| \leq C' \left( \log \sup_{z\in D(z_0, 2\sqrt{2} \max(s,t))} |f(z)| + \left| \log |f(z_0)| \right| + \left| \log |I_1| \right|  \right)\\
\left| \log |f(x'+iy)| \right| \leq C' \left( \log \sup_{z\in D(z_0, 2\sqrt{2}\max(s,t))} |f(z)| + \left| \log |f(z_0)| \right| + \left| \log |I_2| \right|  \right).
\end{aligned}
\end{equation*}
\end{lemme}
\end{tcolorbox}

\section{Proof of Theorem \ref{th:FewRes2}}\label{sec:ProofFew}
We now proceed with the proof of Theorem \ref{th:FewRes2}. In section \ref{sec:General}, our operators $U_N(z)$ were defined for $\Im z< Y_0$. To lighten notations, we will always take $Y_0=+\infty$, which is always the case when considering quantum graphs. The case $Y_0$ finite doesn't add any essential modification to the proof, as our complex-analytic constructions always take place below the real axis, or slightly above it.

\begin{proof}[Proof of Theorem \ref{th:FewRes2}]
\textcolor{black}{Without loss of generality, we may suppose that $a_3 > 2Y_1$, with $Y_1$ as in Hypothesis \ref{Hyp:GenRes}.}

\textbf{Step 1: Some estimates on determinants}
On $\C \backslash \R$, we may write
\begin{align*}
 \mathrm{\det} (\mathrm{Id} - U_N) =  \mathrm{\det} (\mathrm{Id} - \widetilde{U}_N) \times  \mathrm{\det} \big[\mathrm{Id} - (\mathrm{Id}-\widetilde{U}_N)^{-1} \textcolor{black}{(U_N- \widetilde{U}_N)}\big],
\end{align*}
so that
\begin{equation*}
z \in \C\backslash\R \text{ is a resonance } \Longleftrightarrow \mathrm{\det} (\mathrm{Id} - K_N(z)) = 0,
\end{equation*}
where 
$$K_N(z):= (\mathrm{Id}-\widetilde{U}_N(z))^{-1} \textcolor{black}{(U_N(z) - \widetilde{U}_N(z)}.$$

Note that if $z$ is such that $\mathrm{Id}-U_N(z)$ is invertible, we have
\begin{equation}\label{eq:InverseK2}
(\mathrm{Id} - K_N(z))^{-1} = \mathrm{Id} - (\mathrm{Id}-U_N(z))^{-1} \textcolor{black}{(\widetilde{U}_N(z)- U_N(z))}.
\end{equation}

Thanks to (\ref{eq:HypResolv}) and (\ref{eq:PropSchatten}), we have, for any $z\in \C \backslash \R$
\begin{equation}\label{eq:BorneDet10}
\begin{aligned}
|\mathrm{\det} (\mathrm{Id} - K_N(z))| &\leq e^{\|K_N(z)\|_1} \\
&\leq \exp\left[\|(\mathrm{Id}-\widetilde{U}_N(z))^{-1}\| \|\widetilde{U}_N(z)- U_N(z)\|_1\right] \\
&\leq  \exp\left[ \frac{O(1)}{|\Im z|} \|\widetilde{U}_N(z)- U(z)\|_1\right]. 
\end{aligned}
\end{equation}

\textcolor{black}{Finally, if  $\Im z=- \frac{a_3}{2}$ and $\Re z\in [a_1, a_2]$, we have, thanks to Hypothesis \ref{Hyp:GenRes} and to the fact  that $\frac{a_3}{2}> Y_1$}
\begin{equation}\label{eq:BorneDetInv10}
\begin{aligned}
\frac{1}{|\mathrm{\det} (\mathrm{Id} - K_N(z))|} &\leq \exp\left[\|(\mathrm{Id}-U_N(z))^{-1}\| \|\widetilde{U}_N(z)- U_N(z)\|_1 \right]\\
&\leq \exp\left[ O(1) \|\widetilde{U}_N(z)- U_N(z)\|_1 \right]. 
\end{aligned}
\end{equation}

\textbf{Step 2: Using Jensen's formula}

\textcolor{black}{Let $(\delta_N)$ be a bounded sequence of positive numbers, and let $x\in [a_1, a_2]$.
We apply (\ref{eq:Jensen10}) to $f(z)= \det(\mathrm{Id}- K_N(z))$ with $z_0= x - i\frac{a_3}{2}$, $r_N = \frac{a_3}{2}- 2 \delta_N$, and $\varepsilon_N := \frac{2\delta_N}{a_3}$, so that $r'_N:= (1+\varepsilon_N) r_N < \frac{a_3}{2}- \delta_N$.}

By (\ref{eq:BorneDet10}), we have, provided that $N$ is large enough,
\begin{equation}\label{eq:UpperLog10}
\begin{aligned}
\log\left[\max_{|z-z_0|= r'_N} |f(z)| \right] &\leq \frac{O(1)}{|\Im z|} \|\widetilde{U}_N(z)- U_N(z)\|_1 \\
&\leq  \frac{O(1)}{\delta_N} \max_{|z-z_0|= r'_N} \left\|\widetilde{U}_N(z)- U_N(z)\right\|_1,
\end{aligned}
\end{equation}

while by (\ref{eq:BorneDetInv10}), we have
\begin{equation}\label{eq:LowerLog10}
\begin{aligned}
 -\log f(z_0)   = O\left(  \|\widetilde{U}_N(z_0)- U_N(z_0)\|_1\right).
 \end{aligned}
 \end{equation}

Combining (\ref{eq:Jensen10}) with (\ref{eq:UpperLog10}), (\ref{eq:LowerLog10}) and the fact that, for $N$ large enough, we have $\frac{1}{\log (1+\varepsilon_N)} \leq \frac{2}{\varepsilon_N} =   \frac{a_3}{\delta_N}$, we obtain that
\begin{equation}\label{eq:JensenRocks40}
\mathcal{N}_{\textcolor{black}{D}(z_0,r_N)}(U_N) =O \left(\delta_N^{-2} \sup_{z\in \textcolor{black}{D}(z_0, r'_N)}\|\widetilde{U}_N(z)- U_N(z)\|_1\right) = O\left(\gamma_N \delta_N^{-2}\right).
\end{equation}

\textbf{Step 3: Covering by balls}
By elementary Euclidean geometry, we may find $c_1>0$ such that, for all $\delta>0$ small enough and all $x \in \R$,
$$\left\{z\in \C ~|~  - a_3 < \Im z < - 2\delta \text{ and } \Re z\in \left[- x- c_1\sqrt{\delta}, x + c_1 \sqrt{\delta}\right] \right\}\subset \textcolor{black}{D\left( x - i \frac{a_3}{2}, \frac{a_3}{2} -\delta\right)}.$$

Hence, writing $x_k = a_1 +  k c_1 \sqrt{\delta_N}$, we have, up to taking $\delta_N$ smaller,
$$\mathcal{R}_{a_1,a_2, - a_3 - 2\delta_N} \subset \bigcup_{k=1}^{\lfloor \frac{a_2-a_1}{c_1 \sqrt{\delta_N}}\rfloor} \textcolor{black}{D\left( x_k - i \frac{a_3}{2}, \frac{a_3}{2}- \delta_N \right)}.$$

Therefore, applying (\ref{eq:JensenRocks40}), we obtain
$$\mathcal{N}_{a_1,a_2, -a_3, - 2\delta_N} (U_N)\leq  \left\lfloor \frac{a_2-a_1}{c_1 \sqrt{\delta_N}}\right\rfloor O\left( \frac{ \gamma_N}{\delta_N^2} \right) = O\left( \frac{\gamma_N}{\delta_N^{5/2}}\right).$$
Replacing $\delta_N$ by $\delta_N/2$ gives us the result.
\end{proof}

\section{Proof of Theorem \ref{th:ResClose}}\label{sec:ProofMain}
Thanks to the \textcolor{black}{last point of Hypothesis \ref{HypRenforcee}} we know that the resonances of $U_N$ and $\widetilde{U}_N$ are symmetric with respect to the imaginary axis. Hence, the number of resonances in a rectangle can be estimated using the number of resonances in rectangles that are symmetric with respect to the imaginary axis.

Therefore, Theorem \ref{th:ResClose} will follow from the following proposition.

\begin{tcolorbox}
\begin{proposition}\label{Prop:ResClose}
Let $(U_N)$ and $(\widetilde{U}_N)$ satisfy Hypotheses \ref{Hyp:GenRes}, \ref{Hyp:TildeU} and \ref{HypRenforcee}.

Let $a>0$. For any $\varepsilon>0$ and any $N\in \N$, we may find $a^+ \in [a+\varepsilon,a+ 2\varepsilon]$ and $a^-\in [a-2\varepsilon, a-\varepsilon]$ such that, for any $c>0$, we have 
$$\mathcal{N}_{-a^\pm, a^\pm, -c, 0}(U_N) - \mathcal{N}_{-a^\pm, a^\pm}(\widetilde{U}_N)=o(N).$$
\end{proposition}
\end{tcolorbox}

Let us now proceed with the proof of Proposition \ref{Prop:ResClose}.
Let $a\in \R$, and let $\varepsilon>0$. All the values of the constants in the sequel will implicitly depend on $a$ and $\varepsilon$. Let $(\xi_N)_N$ be a sequence decreasing to zero. Its exact value will be chosen later on.

For each $x\in \R$, the operator $\widetilde{U}_N(x)$ is unitary, so it can be diagonalised in an orthonormal basis. Let us denote by $\left(e^{i\theta_k(x)}\right)_{1\leq k\leq d_N}$ its eigenvalues. The $\theta_k$ are parametrised so that for each $k$, $x\mapsto \theta_k(x)$ is smooth; let $u_k(x)$ be an associated normalized eigenvector.

For each $x\in \R$, we write
$$\mathcal{K}_N(x) := \left\{1\leq  k \leq d_N \text{ such that } |e^{i\theta_k(x)}-1| \leq \xi_N \right\}.$$

\begin{tcolorbox}
\begin{lemme}\label{lem:SmallTunnels3}
Let $\varepsilon>0$. There exists $a^+\in [a+\varepsilon, a+ 2\varepsilon]$ and $a^-\in [a- 2\varepsilon, a- \varepsilon]$ depending of $N$, and $c'_1>0$ independent on $N$ such that
\begin{equation}\label{eq:Tunnels10}
|\mathcal{K}_N(a^\pm)| \leq C(\varepsilon) \frac{N}{|\log \xi_N|}.
\end{equation}
\end{lemme}
\end{tcolorbox}

\begin{proof}
We will only explain the construction of $a^+$, as the construction of $a^-$ is exactly the same.
We apply the results of section \ref{sec:Max} for $z_0= a + \frac{3\varepsilon}{2} + i \frac{\varepsilon}{4}$, $s=t = \frac{\varepsilon}{2}$ and $\alpha=\frac{1}{\sqrt{2}}$, with $f(z) := \det(\mathrm{Id} - U_N(z))$.

Lemma \ref{lem:GoodPoint} gives us the existence of $a^+\in [a+ \varepsilon, a + 2\varepsilon]$ such that
$$|\log f(a^+)| = O\left( 1+ \log \sup_{z\in D(z_0, \sqrt{2}\varepsilon )} |f(z)| + \left| \log |f(z_0)| \right| \right).$$

Now, thanks to (\ref{eq:TraceBound}) and (\ref{eq:PropSchatten1}), we have $\log \sup_{z\in D(z_0, \sqrt{2} \varepsilon)} |f(z)|= O(N)$. On the other hand, we have 
\begin{align*}
\frac{1}{|\det(\mathrm{Id} - U_N(z_0))|} &= \left|\det \left[ \left(\mathrm{Id}- U_N(z_0)\right)^{-1}\right]\right|\\
&= \left|\det \left[ \mathrm{Id} + \left(\mathrm{Id}-U_N(z_0)\right)^{-1} U_N(z_0)\right]\right|\\
&\leq \exp\left[ \left\| \left(\mathrm{Id}-U_N(z_0)\right)^{-1} U_N(z_0) \right\|_1 \right]\\
&\leq \exp\left[ \left\| \left(\mathrm{Id}-U_N(z_0)\right)^{-1}\left\| \right\| U_N(z_0) \right\|_1 \right],
\end{align*}

so that, thanks to (\ref{eq:HypResolv0}) and   (\ref{eq:TraceBound}), we have
$$\left| \log |f(z_0)| \right| =O(N).$$

Therefore, we have
$$\left| \log |f(a^+)| \right| =O(N).$$

In particular,
$$|\det (U_N (a^+) - \mathrm{Id})| \geq e^{-CN}.$$

Now, we have 
\begin{align*}
|\det (U_N (a^+) - \mathrm{Id})| = \prod_{k=1}^{d_N} |1- e^{i\theta_k(a^+)}| &\leq 2^{d_N} \left( \prod_{\underset{k\in \mathcal{K}_N(x)}{k=1}}^N |1- e^{i\theta_k(\textcolor{black}{a^+})}| \right) \\&\leq C^N \xi_N^{|\mathcal{K}_N(\textcolor{black}{a^+})|}.
\end{align*}

Therefore, taking logarithms, we have
$$ -CN\leq  C' N  + |\mathcal{K}_N(a^+)| \log \xi_N,$$
so that

$$|\mathcal{K}_N(a^+)| =O\left(\frac{N}{|\log \xi_N|}\right).$$
\end{proof}

Recalling that the resonances of $U_N$ are isolated, we may do an arbitrarily small perturbation of $a^\pm$ from the previous lemma such that (\ref{eq:Tunnels10}) still holds, and we have
\begin{equation}\label{eq:NotARes}
\forall y \in \R, a^\pm + iy~~~~ \text{is not a resonance of $U_N$},
\end{equation}

From now on, we fix $a^\pm$ such that (\ref{eq:Tunnels10}) and (\ref{eq:NotARes}) hold. We write
$$V_N^\pm := \mathrm{Vect} \left(\left\{ u_k(-a^\pm) \text{ with } k \in \mathcal{K}_N(-a^\pm) \right\}\right).$$

We define a unitary matrix $\Sigma^\pm_N$ of size $d_N$ by
$$\Sigma^\pm_N = \begin{cases}
\mathrm{Id} &\text{ on } \left(V_N^\pm\right)^\perp\\
e^{3i \xi_N} \mathrm{Id}  &\text{ on } V_N^\pm.
\end{cases}$$

Note that we have $\|\Sigma_N^\pm -\mathrm{Id}\| = O(\xi_N)$ and $\|\Sigma_N^\pm -\mathrm{Id}\|_1 = O\left(\frac{N \xi_N}{|\log \xi_N|}\right)$.
Finally, we write
$$\widetilde{U}_N^\pm(z) := \Sigma^\pm_N \widetilde{U}_N(z).$$
Hence, $\widetilde{U}_N^\pm(z)$ is a holomorphic family of matrices, and for any $z$ in a compact set, there exists $C$ such that
\begin{equation}\label{eq:TaillePerturbNorme}
\|\widetilde{U}_N^\pm(z)- \widetilde{U}_N(z)\|\leq C \xi_N
\end{equation}
\begin{equation}\label{eq:TaillePerturb}
\|\widetilde{U}_N^\pm(z)- \widetilde{U}_N(z)\|_1 \leq C \frac{N \xi_N}{|\log \xi_N|}.
\end{equation}

Furthermore, if we denote by $\theta^{\pm,N}_k(z)$ the eigenvalues of $\widetilde{U}_N^\pm(z)$, we have
\begin{equation}\label{eq:SpectralGap}
\forall k=1,..., d_N,~~ \left| e^{i \theta^{\pm,N}_k(-a^\pm)} - 1 \right| \geq \xi_N.
\end{equation}

Recall that, if $a_1<a_2\in \R$, we denote by 
$\mathcal{N}_{a_1,a_2}(\widetilde{U}_N)$ the number of resonances of $\widetilde{U}_N$ in $[a_1,a_2]$ counted with multiplicity. We will also denote by $\mathcal{N}_{a_1,a_2}(\widetilde{U}^\pm_N)$  the number of zeroes of $\det\left( \mathrm{Id} - \widetilde{U}^\pm(x)\right)$ for $x\in (a_1,a_2)$, counted with multiplicities. Proposition \ref{Prop:ResClose} will follow from the next two lemmas.

\begin{tcolorbox}
\begin{lemme}\label{lem:MemeNombreVP}
We have
$$\mathcal{N}_{-a^\pm,a^\pm}(\widetilde{U}^\pm_N) - \mathcal{N}_{-a^\pm,a^\pm}(\widetilde{U}_N)=o(N). $$
\end{lemme}
\end{tcolorbox}

\begin{tcolorbox}
\begin{lemme}\label{lem:ContourBorne}
We may find $c^\pm>0$ independent of $N$ such that 
$$ \mathcal{N}_{-a^\pm, a^\pm, -c^\pm \xi_N, 0} (U_N) - \mathcal{N}_{-a^\pm,a^\pm}(\widetilde{U}^\pm_N) = o(N).$$
\end{lemme}
\end{tcolorbox}

Indeed, we have 
\begin{align*}
\left|\mathcal{N}_{-a^\pm, a^\pm, -c, 0}(U_N) - \mathcal{N}_{-a^\pm, a^\pm}(\widetilde{U}_N) \right|
&\leq \left| \mathcal{N}_{-a^\pm, a^\pm, -c, 0}(U_N) -  \mathcal{N}_{-a^\pm, a^\pm, -c^\pm \xi_N, 0}(U_N)\right| \\
&+ \left| \mathcal{N}_{-a^\pm, a^\pm, -c^\pm \xi_N, 0} (U_N) - \mathcal{N}_{-a^\pm,a^\pm}(\widetilde{U}^\pm_N) \right|\\
&+ \left|\mathcal{N}_{-a^\pm,a^\pm}(\widetilde{U}_N)- \mathcal{N}_{-a^\pm,a^\pm}(\widetilde{U}^\pm_N)) \right|
\end{align*}

The first term is smaller than  $\mathcal{N}_{-a^\pm, a^\pm, -c, -\frac{c^\pm}{2} \xi_N}(U_N)$, which is a $o(N)$ thanks to Theorem \ref{th:FewRes2}. The second term is a $o(N)$ thanks to Lemma \ref{lem:ContourBorne}, while the last term is a $o(N)$ thanks to Lemma \ref{lem:MemeNombreVP}. Proposition \ref{Prop:ResClose} follows.

Let us start with the proof of Lemma \ref{lem:MemeNombreVP}.

\begin{proof}[Proof of Lemma \ref{lem:MemeNombreVP}]
For each $1\leq k \leq d_N$, we denote by $\mathcal{N}_k(a_1,a_2)$ (resp. $\mathcal{N}_k^\pm(a_1,a_2)$) the number of zeroes of $e^{i\theta_k(x)}-1$ (resp. $e^{i\theta_k^\pm(x)} -1$) for $x\in (a_1,a_2)$.

To lighten notations, in all the proof, we will write $a$ instead of $a^\pm$.

By assumption, $\widetilde{U}_N(-a)$ and $\left(\widetilde{U}_N(a)\right)^*$ are \textcolor{black}{similar}, so, for each $k$, there exists $n_k\in \Z$ such that
$$\theta_k (a) = - \theta_k(-a) + 2 n_k\pi.$$

Without loss of generality, we will suppose that $\theta_k(-a) \in [0,2\pi[$. 
Differentiating the relation $\langle u_k, \widetilde{U}_N(x) u_k(x) \rangle = e^{i\theta_k(x)}$, we obtain 
\begin{equation*}
i \theta_k'(x) e^{i\theta_k(x)} = \langle u_k(x), \frac{\mathrm{d} \widetilde{U}_N(x)}{\mathrm{d}x} u_k(x) \rangle.
\end{equation*}

In other words, we have
\begin{equation*}
\theta_k'(x) = -i  \langle u_k(x), \frac{\mathrm{d} \widetilde{U}_N(x)}{\mathrm{d}x} \widetilde{U}_N(x) u_k(x) \rangle.
\end{equation*}

In particular, thanks to (\ref{eq:Positivite}), we see that this quantity is positive.

Since $\theta_k(x)$ is an increasing function, we see that
$$\mathcal{N}_k(-a, a)= \begin{cases}
n_k &\text{ if } \theta_k(-a)\in [0,\pi[\\
n_k +1 &\text{ if } \theta_k(-a)\in [\pi,2\pi[.
\end{cases}
$$

On the other hand, we have
$$n_k = \frac{1}{2\pi} \left(\theta_k (a) + \theta_k(-a)\right) = \frac{\theta_k(-a)}{\pi} + \frac{1}{2\pi} \int_{-a}^a \frac{\mathrm{d}\theta_k(x)}{\mathrm{d}x} \mathrm{d}x.$$

All in all, we have
\begin{align*}
\mathcal{N}_{-a,a}(\widetilde{U}_N) &= \frac{1}{\pi} \sum_k \theta_k(-a)  + \frac{1}{2\pi} \int_{-a}^a \sum_k \frac{\mathrm{d}\theta_k(x)}{\mathrm{d}x} \mathrm{d}x + \sharp \left\{ k \text{ such that }  \theta_k(-a)\in [\pi,2\pi[ \right\}\\
&=  \frac{1}{\pi} \sum_k \theta_k(-a)  - \frac{i}{2\pi} \int_{-a}^a \mathrm{Tr}\left[\frac{\mathrm{d}\widetilde{U}_N(x)}{\mathrm{d} x} \widetilde{U}_N^{-1}(x) \right] \mathrm{d}x + \sharp \left\{ k \text{ such that }  \theta_k(-a)\in [\pi,2\pi[ \right\},
\end{align*}

Recall that
\begin{equation}\label{eq:ThetaPrime}
\theta_k^\pm(-a) = \begin{cases}
\theta_k(-a) \text{ if } |e^{i\theta_k(-a)}-1|> \xi_n\\
\theta_k(-a)+ 3\xi_n \text{ if } |e^{i\theta_k(-a)}-1|\leq \xi_n.
\end{cases}
\end{equation}

Noting that $ \frac{\mathrm{d}\widetilde{U}_N^\pm(x)}{\mathrm{d} x} \left(\widetilde{U}_N^\pm(x)\right)^{-1}  = \frac{\mathrm{d}\widetilde{U}_N(x)}{\mathrm{d} x} \widetilde{U}_N^{-1}(x) $, the same argument as above tells us that 
\begin{align*}
\mathcal{N}_{-a,a}(\widetilde{U}^\pm_N)
&=  \frac{1}{\pi} \sum_k \theta^\pm_k(-a)  - \frac{i}{2\pi} \int_{-a}^a \mathrm{Tr}\left[\frac{\mathrm{d}\widetilde{U}_N(x)}{\mathrm{d} x} \widetilde{U}_N^{-1}(x) \right] \mathrm{d}x  + \sharp \left\{ k \text{ such that } \textcolor{black}{ \left(\theta^\pm_k(-a) ~\mathrm{mod} ~ \pi\right)} \in [\pi,2\pi[ \right\}.
\end{align*}

Now, from (\ref{eq:ThetaPrime}) and (\ref{eq:Tunnels10}), we see that 
$$\sharp \left\{ k \text{ such that }   \textcolor{black}{ \left(\theta^\pm_k(-a) ~\mathrm{mod} ~ \pi\right)} \in [\pi,2\pi[ \right\}= \sharp \left\{ k \text{ such that }  \theta_k(-a)\in [\pi,2\pi[ \right\} + O \left(\frac{N}{|\log \xi_N|}\right),$$
 while $\sum_k \theta^\pm_k(-a) = \sum_k \theta_k(-a) + O\left(\frac{N \xi_N}{|\log \xi_N|}\right)$. The result follows.
\end{proof}

We may now proceed with the proof of Lemma \ref{lem:ContourBorne}.

\begin{proof}[Proof of Lemma \ref{lem:ContourBorne}]

We will only bound the quantity 
$$\left| \mathcal{N}_{-a^+, a^+, -c^+ \xi_N, 0} (U_N) -\mathcal{N}_{-a^+,a^+}(\widetilde{U}^+_N)\right|,$$
 since the proof for $\left| \mathcal{N}_{-a^-, a^-, -c^- \xi_N, 0} (U_N) -\mathcal{N}_{-a^-,a^-}(\widetilde{U}^-_N)\right|$  is exactly the same.

We know from (\ref{eq:SpectralGap}) that
\begin{equation}\label{eq:TildeUSansVP50}
\left\| \left( \mathrm{Id}- \widetilde{U}_N^+(\pm a^+)\right)^{-1} \right\| \leq \frac{1}{\xi_N}.
\end{equation}

Let $\mathbf{z}\in \C$ with $|\mathbf{z}|\leq 1$.
We have
\begin{equation}\label{eq:ThroughReal20}
\mathrm{Id} - \widetilde{U}_N(\pm a^+ + \mathbf{z}) = \left(\mathrm{Id} - \widetilde{U}_N(\pm a^+)\right) \left[ \mathrm{Id} + \left(\mathrm{Id} - \widetilde{U}_N(\pm a^+)\right)^{-1} \left(\textcolor{black}{\widetilde{U}_N(\pm a^+)} - \widetilde{U}_N(\pm a^+ + \mathbf{z}) \right)\right].
\end{equation}

Note that by (\ref{eq:BorneGlobaleDerivU}), there exists $\mathrm{C}_1'$ such that
$$\left\|\widetilde{U}_N(\pm a^+ + \mathbf{z}) - \widetilde{U}_N(\pm a^+) \right\| \leq \mathrm{C}'_{1}|\mathbf{z}|.$$
Therefore, thanks to (\ref{eq:TildeUSansVP50}), if $|\mathbf{z}|\leq \frac{\xi_N}{2 \mathrm{C}'_1}$,  we have 
$$ \left\| \left(\mathrm{Id} - \widetilde{U}_N(\pm a^+)\right)^{-1} \left(\textcolor{black}{\widetilde{U}_N(\pm a^+)} - \widetilde{U}_N(\pm a^+ + \mathbf{z}) \right)\right\| \leq \frac{1}{2}.$$
 Hence, we deduce from (\ref{eq:ThroughReal20}) that 
$$ \left\| \left[\mathrm{Id} - \widetilde{U}_N(\pm a^+ + \mathbf{z}) \right]^{-1} \right\| \leq \frac{2}{\xi_N}.$$

All in all, we may find $c'_2>0$ such that $\det \left( \widetilde{U}_N^+(z) - \mathrm{Id} \right)$ does not vanish in $B\left(\pm a^+ + i \frac{c'_2}{8} \xi_N, \frac{c'_2}{2} \xi_N \right)$, and that we have
\begin{equation}\label{eq:BorneInverse10}
\forall z \in B\left(\pm a^+ + i \frac{c'_2}{8} \xi_N, \frac{c'_2}{2} \xi_N \right), ~~ \left\|\left(\widetilde{U}_N^+(z) - \mathrm{Id}\right)^{-1}\right\| \leq \frac{2}{\xi_N}.
\end{equation}

On the other hand, thanks to (\ref{eq:DefDiffTrace}) and (\ref{eq:TaillePerturb}), we have for all $z$ in a compact set,
\begin{equation}\label{eq:TraceTilde10}
  \left\| U_N(z) - \widetilde{U}_N^+(z) \right\|_1 \leq   \left\| U_N(z) - \widetilde{U}_{\textcolor{black}{N}}(z) \right\|_1 +   \left\| \widetilde{U}_N(z) - \widetilde{U}_N^+(z) \right\|_1 \leq   C \zeta_N \xi_N,
\end{equation}
where we write 
$$\zeta_N := \frac{N}{|\log \xi_N|} +  \frac{\gamma_N}{\xi_N}= o(N).$$

Let us set
\begin{align*}
K_N^+(z)&:= (\mathrm{Id}-\widetilde{U}_N^+(z))^{-1} \textcolor{black}{\left(U_N(z) - \widetilde{U}_N^+(z)\right)}\\
D_N^+(z) &:= \det (\mathrm{Id} - K_N^+(z))
\end{align*}
so that
\begin{equation*}
z \in \C\backslash\R \text{ is a resonance of $U_N$} \Longleftrightarrow D_N^+(z) = 0.
\end{equation*}

Thanks to equations (\ref{eq:PropSchatten}), (\ref{eq:PropSchatten1}) and (\ref{eq:TraceTilde10}), we have, for all $z$ in a compact set
\begin{equation}\label{eq:BorneDet20}
\begin{aligned}
 |D_N^+(z)| &\leq e^{\|K_N^+(z)\|_1}\\
& \leq \exp\left[ \left\|(\mathrm{Id}-\widetilde{U}_N^+(z))^{-1}\right\| \left\|\widetilde{U}_N^+(z)- U_N(z)\right\|_1 \right] \\
 &\leq  \exp\left[ C \zeta_N \xi_N  \left\|(\mathrm{Id}-\widetilde{U}_N^+(z))^{-1}\right\| \right]. 
\end{aligned}
\end{equation}

In particular,
\begin{itemize} 
\item \textcolor{black}{If $\Im z \neq 0$, we have $\left\|(\mathrm{Id}-\widetilde{U}_N^+(z))^{-1}\right\| = \left\|\left((\Sigma_N^+)^{-1}-\widetilde{U}_N(z)\right)^{-1}\right\| = O (|\Im z|^{-1})$ thanks to  (\ref{eq:BorneSurLesNormes}) and to the fact that $\Sigma_N^+$ is an isometry. Therefore, we have
\begin{equation}\label{eq:BorneDet22}
\begin{aligned}
 \log |D_N^+(z)|= O(\zeta_N \xi_N |\Im z|^{-1}). 
\end{aligned}
\end{equation}}

\item If  $z\in \textcolor{black}{D}\left(\pm a^+ + i \frac{c'_2}{8} \xi_N, \frac{c'_2}{2} \xi_N \right)$, we see from \textcolor{black}{(\ref{eq:BorneInverse10})} that
\begin{equation}\label{eq:BorneDet21}
\begin{aligned}
 \log |D^+_N(z)|= O(\zeta_N). 
\end{aligned}
\end{equation}
\end{itemize}

On the other hand, if $z$ is such that $\mathrm{Id}- U_N(z)$ is invertible, then we have
\begin{equation*}
\left(\mathrm{Id} - K_N^{\textcolor{black}{+}}\left( z\right)\right)^{-1} = \mathrm{Id} - \left(\mathrm{Id}-U_N\left(z\right)\right)^{-1}  \textcolor{black}{\left(\widetilde{U}^{+}_N\left(z\right) - U_N\left(z\right)\right)},
\end{equation*}
so that 
\begin{equation}\label{eq:BorneDetInv21}
\begin{aligned}
\frac{1}{\left|D_N^+(z)\right|} &\leq \exp\left[\left\|\left(\mathrm{Id}-U_N\left(z \right)\right)^{-1}\right\| \left\|\widetilde{U}^{\textcolor{black}{+}}_N\left(z\right)- U_N\left(z\right)\right\|_1 \right]\\
&\leq \exp\left[ C \left\|\left(\mathrm{Id}-U_N\left(z \right)\right)^{-1}\right\|   \xi_N\zeta_N \right].
\end{aligned}
\end{equation}

An application of Jensen's formula (\ref{eq:Jensen10}) along with (\ref{eq:BorneDet21}), (\ref{eq:BorneDetInv21}) and (\ref{eq:HypResolv0})  tells us that
\begin{equation}\label{eq:FewResSquares10}
\mathcal{N}_{\textcolor{black}{D}\left(\pm a^+ + i \frac{c'_2}{8} \xi_N, \frac{c'_2}{4} \xi_N \right)} (U_N) \leq C \zeta_N
\end{equation}

We write 
$$\mathcal{S}_{z, r} := \mathcal{R}_{\Re z - r, \Re z + r, \Im z - r, \Im z + r}$$

Let $c'_3$ be such that 
$$\mathcal{S}_{a^+ + i \frac{c'_3}{4} \xi_N ,2c'_3 \xi_N} \subset \textcolor{black}{D}\left(a^+ + i \frac{c'_2}{8} \xi_N, \frac{c'_2}{4} \xi_N \right).$$

\textcolor{black}{Taking $Y_1$ as in Hypothesis \ref{Hyp:GenRes}, we define}
\begin{align*}
W^\pm &:= \mathcal{S}_{\pm a^+ + i \frac{c'_3}{4} \xi_N, 2c'_3 \xi_N} ~~~~ &W^\pm_1 &:= \mathcal{S}_{\pm a^+ + i \frac{c'_3}{4} \xi_N, \frac{c'_3}{2} \xi_N}\\
W^0 &:= \mathcal{R}_{- a^+ - c'_3\xi_N, a^+ +c'_3\xi_N, - 2\textcolor{black}{Y_1} + \frac{c'_3}{16}\xi_N, -\frac{c'_3}{16}\xi_N} ~~ &W^0_1 &:=\mathcal{R}_{- a^+ - \frac{c'_3}{2}\xi_N, a^+ +\frac{c'_3}{2}\xi_N, \textcolor{black}{-2Y_1 + \frac{c'_3}{8}\xi_N} , -\frac{c'_3}{8}\xi_N}\\
\hat{W}^0 &:= \mathcal{R}_{- a^+ - c'_3\xi_N, a^+ +c'_3\xi_N,  \frac{c'_3}{16}\xi_N, 2\textcolor{black}{Y_1}  - \frac{c'_3}{16}\xi_N} ~~ & \hat{W}^0_1 &:=\mathcal{R}_{- a^+ - \frac{c'_3}{2}\xi_N, a^+ +\frac{c'_3}{2}\xi_N, \frac{c'_3}{8}\xi_N, \textcolor{black}{2Y_1} - \frac{c'_3}{8}\xi_N }\\
W&:= \textcolor{black}{W^+ \cup W^- \cup W^0 \cup \hat{W}^0} ~~ &W_1&:= \textcolor{black}{W_1^+ \cup W_1^- \cup W_1^0 \cup \hat{W}_1^0}.
\end{align*}

Thanks to Hypothesis \ref{Hyp:GenRes}, we have $\mathcal{N}_{\textcolor{black}{\hat{W}^0}} (U_N) = 0$.
Equation (\ref{eq:FewResSquares10}) tells us that $\mathcal{N}_{W^\pm} (U_N) \leq C \zeta_N $, while Theorem \ref{th:FewRes2} implies that $\mathcal{N}_{\textcolor{black}{W^0}} (U_N) \leq C  \frac{\gamma_N}{\xi_N^{5/2}}$. All in all, we have
\begin{equation}\label{eq:ResW}
\mathcal{N}_W (U_N) \leq C \zeta_N + C  \frac{\gamma_N}{\xi_N^{5/2}}= o(N),
\end{equation}
provided we take $\xi_N$ going to zero slowly enough.

We denote by $(z_j)_{j=1,...,J_N} :=\mathrm{Res} (U_N)\cap W$, so that $J_N \leq C \zeta_N + C  \frac{C \gamma_N}{\xi_N^{5/2}}$. 

 We set
\begin{equation*}
\mathcal{B}^+_N(z):= \prod_{j\textcolor{black}{=}1}^{J_N} (z-z_j), ~~~~  G^+_N(z) :=  \frac{D_N^+(z)}{\mathcal{B}^+_N(z)}.
\end{equation*}

We are now going to apply the results of section \ref{sec:Max} to estimate $|\log G_N^+|$ in $W_{\textcolor{black}{1}}$. We will thus take $\mathcal{K}=W$, and apply \textcolor{black}{Lemma \ref{Lem:BorneDerivG}} in various rectangles.

\begin{itemize}
\item In $W_1^\pm$, we take $z_0 = \pm a^+ + i \frac{c'_3}{4} \xi_N$, $s=t= 2 c_3' \xi_N$, and $\alpha=\frac{1}{2}$, so that $W^\pm= \mathbf{R}_{z_0, s,t}$ and $W_1^\pm = \mathbf{R}_{z_0, \alpha^2 s, \alpha^2t}$.
\item  In $W_1^0$, we take $z_0 = - i \textcolor{black}{Y_1} $, $s= a^+ + c_3' \xi_N$, $t= \textcolor{black}{Y_1}  -\frac{c_3'}{16} \xi_N$, so that $W^0=\mathbf{R}_{z_0, s,t}$. Let $\alpha_1 = \left( \frac{a^+ + c'_3 \xi_N}{a^+ + \frac{c'_3}{2} \xi_N}\right)^{1/2}$, $\alpha_2 = \left( \frac{\textcolor{black}{Y_1}  - \frac{c'_3}{16} \xi_N}{\textcolor{black}{Y_1} - \frac{c'_3}{8} \xi_N}\right)^{1/2}$, and $\alpha = \min (\alpha_1, \alpha_2)$. This way, we have $0<\alpha<1$, and $1-\alpha = \Theta(\xi_N)$. Furthermore, we have $W_1^0\subset \mathbf{R}_{z_0, \alpha^2 s, \alpha^2t}$

\item In $\hat{W}_1^0$, take $\alpha$, $s$ and $t$ as above, with $z_0 = \textcolor{black}{i} (\textcolor{black}{Y_1} + \frac{c_3'}{16} \xi_N)$.
\end{itemize}

Recall that (\ref{eq:BorneDerivG}) tells us that, with $f=G_N^+$, we have  for all $z \in \mathbf{R}_{z_0, \alpha^2 s, \alpha^2t}$
$$\left| \partial_z \log G_N^+(z) \right| \leq \frac{C}{r_m(1-\alpha)}  \left(1+ \mathrm{Diam}(W) +  |\log (1-\alpha) r_m| \right) \left(\mathcal{N}_W (U_N)  + \log \sup_{z\in \mathbf{R}_{z_0, s,t}} |D_N^+(z)| + \left|\log |D_N^+(z_0)| \right| \right). $$

Here, $ \mathrm{Diam}(W)$ is bounded independently of $N$. In all the situations above, $(1-\alpha) r_M$  is either a $\Theta(\xi_N)$ or a $\Theta(\xi_N^2)$, and $\mathcal{N}_W (U_N)$ is bounded by (\ref{eq:ResW}). Concerning the last two terms:

\begin{itemize}
\item In $\textcolor{black}{W^\pm}$, $\left|\log |D_N^+(z_0)| \right|$ can be estimated by (\ref{eq:BorneDetInv21}) along with (\ref{eq:HypResolv0}), and we see that $\left|\log |D_N^+(z_0)|\right|= O(\zeta_N).$ The term $\log \sup_{z\in \mathbf{R}_{z_0, s,t}} |D_N^+(z)|$ is also a $O(\zeta_N)$ thanks to (\ref{eq:BorneDet21}).
\item In $\textcolor{black}{\hat{W}^0}$,  $\left|\log |D_N^+(z_0)| \right|$ can also estimated by (\ref{eq:BorneDetInv21}) along with (\ref{eq:HypResolv0}), and we see that $\left|\log |D_N^+(z_0)|\right| =O(\xi_N \zeta_N).$ As for $\log \sup_{z\in \mathbf{R}_{z_0, s,t}} |D_N^+(z)|$, it is a \textcolor{black}{$O(\zeta_N)$} thanks to (\ref{eq:BorneDet22}).
\item In $\textcolor{black}{W^0}$, $\log \sup_{z\in \mathbf{R}_{z_0, s,t}} |D_N^+(z)|$ is also a \textcolor{black}{$O(\zeta_N)$} thanks to (\ref{eq:BorneDet22}). As for $\left|\log |D_N^+(z_0)| \right|$, we estimate it with (\ref{eq:BorneDetInv21}) \textcolor{black}{along with the third point of Hypothesis \ref{Hyp:GenRes}} to obtain that it is a $O(\xi_N \zeta_N)$.
\end{itemize}

All in all, we obtain that, for all $z\in W$, we have
\begin{equation*}
\left| \partial_z \log G_N^+(z) \right| \leq \frac{C}{\xi_N^2} (1+ |\log(\xi_N)|)   \left(\zeta_N + \frac{\gamma_N}{\xi_N^{5/2}} \right). 
\end{equation*}

Recalling that $\zeta_N= \frac{N}{|\log \xi_N|} +  \frac{\gamma_N}{\xi_N}$ and that $\gamma_N = o(N)$, we see that, if we take $\xi_N$ going to zero slowly enough, we have 
\begin{equation}\label{eq:BretzelDort}
\sup_{z\in W} \left| \partial_z \log G_N^+(z) \right| = o(N).
\end{equation}

Now, for each $\eta\in [0,1/2]$, write
$$\Omega_\eta := \mathcal{R}_{-a^+, a^+, -c_3'\left( \frac{1 - \eta}{8}\right) \xi_N,  c_3'\left( \frac{1 - \eta}{8}\right) \xi_N},$$
so that $\partial \Omega_\eta \subset W$. Recalling (\ref{eq:NotARes}) and the fact that the resonances of $U_N$ are isolated, we can find, for each $N$, an $\eta_N\in [0, 1/2]$ such that $U_N$ has no resonances on $\partial \Omega_{\eta_N}$.

Therefore, we may apply the residue formula to obtain
\begin{equation*}
\begin{aligned}
\mathcal{N}_{\Omega_0} (U_N)\leq \mathcal{N}_{\Omega_{\eta_N}} (U_N) &= \frac{1}{2i\pi}\int_{\partial \Omega_{\eta_N}} \frac{1}{\det(\mathrm{Id}-U_N(z)) } \frac{\mathrm{d}}{\mathrm{d}z}\det(\mathrm{Id}-U_N(z))\mathrm{d}z\\
&=  \frac{1}{2i\pi}\int_{\partial \Omega_{\eta_N}}  \partial_z \left[\log\det(\mathrm{Id}-U_N(z))\right] \mathrm{d}z\\
&= \frac{1}{2i\pi}\int_{\partial \Omega_{\eta_N}}  \partial_z \left[\log\det(\mathrm{Id}-\widetilde{U}_N^+(z)) + \log\det(\mathrm{Id}- K_N^+(z)) \right] \mathrm{d}z\\
&=\mathcal{N}_{-a^+,a^+}(\widetilde{U}^+_N) + \frac{1}{2i\pi}\int_{\partial \Omega_{\eta_N}} \partial_z \log G^+_N(z) \mathrm{d}z + \frac{1}{2i\pi}\int_{\partial \Omega_{\eta_N}} \partial_z \log \mathcal{B}^+_N(z) \mathrm{d}z
\end{aligned}
\end{equation*}

Now, by (\ref{eq:BretzelDort}), we have
$$\left|\frac{1}{2i\pi}\int_{\partial \Omega_{\eta_N}} \partial_z \log G^+_N(z) \mathrm{d}z \right| =o(N).$$

On the other hand, $\log\mathcal{B}^+_N(z)= \sum_{j=1}^{J_N} \log (z-z_j)$, so, by the residue formula
$$\frac{1}{2i\pi}\int_{\partial \Omega_{\eta_N}} \partial_z \log \mathcal{B}^+_N(z) \mathrm{d}z = \mathcal{N}_{\Omega_{\eta_N}} (U_N) = o(N),$$
thanks to (\ref{eq:ResW}).

This shows that 
$$\left| \mathcal{N}_{\Omega_0} (U_N) -\mathcal{N}_{-a^+,a^+}(\widetilde{U}^+_N) \right| = o(N),$$
and the lemma follows.
\end{proof}

\appendix

\section{Irrational quantum graphs with real resonances}\label{app:Irr}

The simplest examples of open quantum graphs having resonances on the real axis are those whose compact part is just a segment, contains a loop (or a part that is equivalent to a loop, see the left side of Figure \ref{exampleTrivial}), or those where the lengths on a cycle are all rationally dependent (see the right part of Figure \ref{exampleTrivial}).

Apart from variants of these simple situations, it is not easy to build examples of quantum graphs having resonances on the real axis. Our aim here is to show that such graphs exist, thanks to the following proposition.

\begin{figure}[h]
    \begin{minipage}[c]{.46\linewidth}
        \centering
   \includegraphics[scale=0.3]{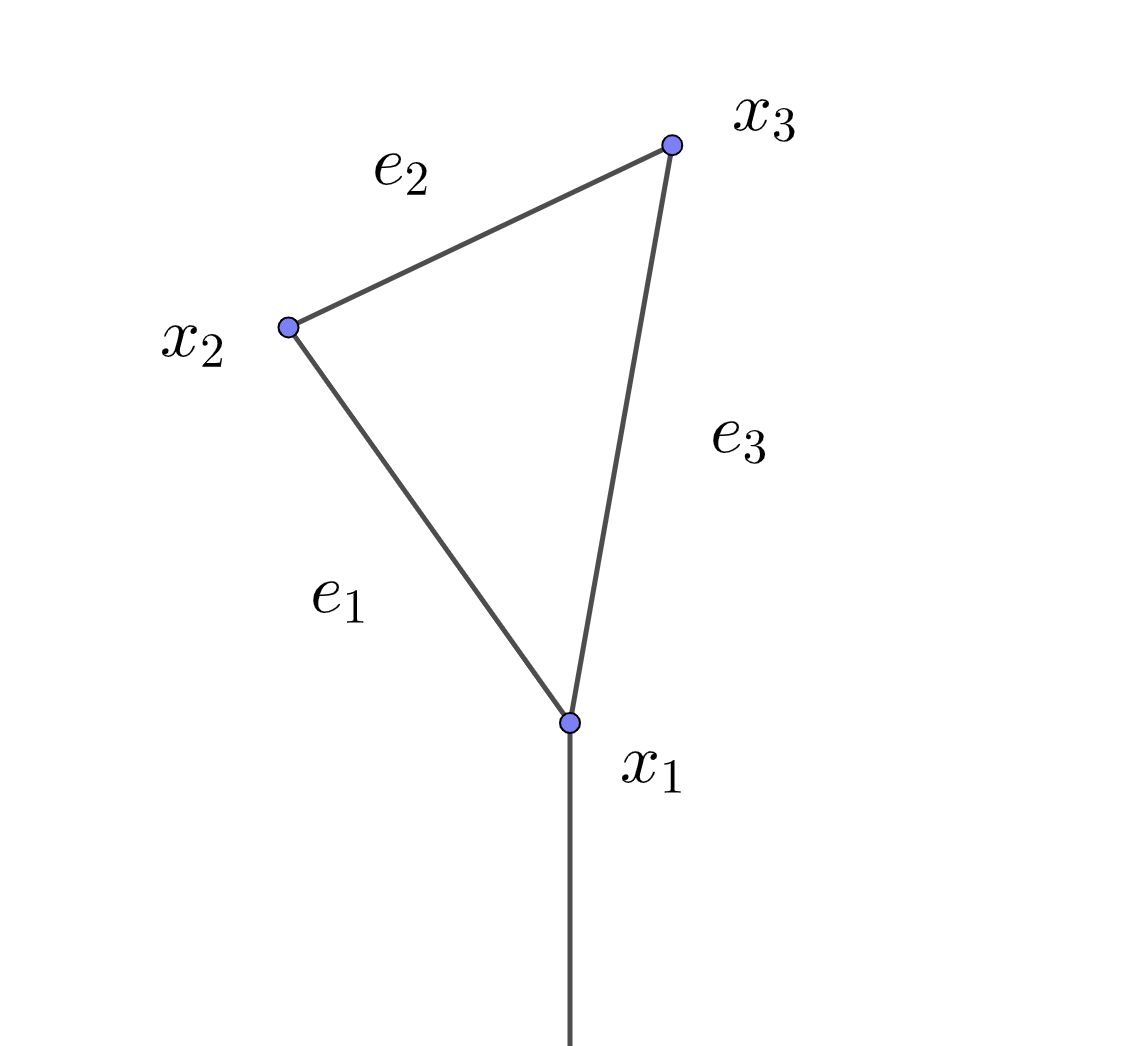}
    \end{minipage}
    \hfill%
    \begin{minipage}[c]{.46\linewidth}
        \centering
   \includegraphics[scale=0.3]{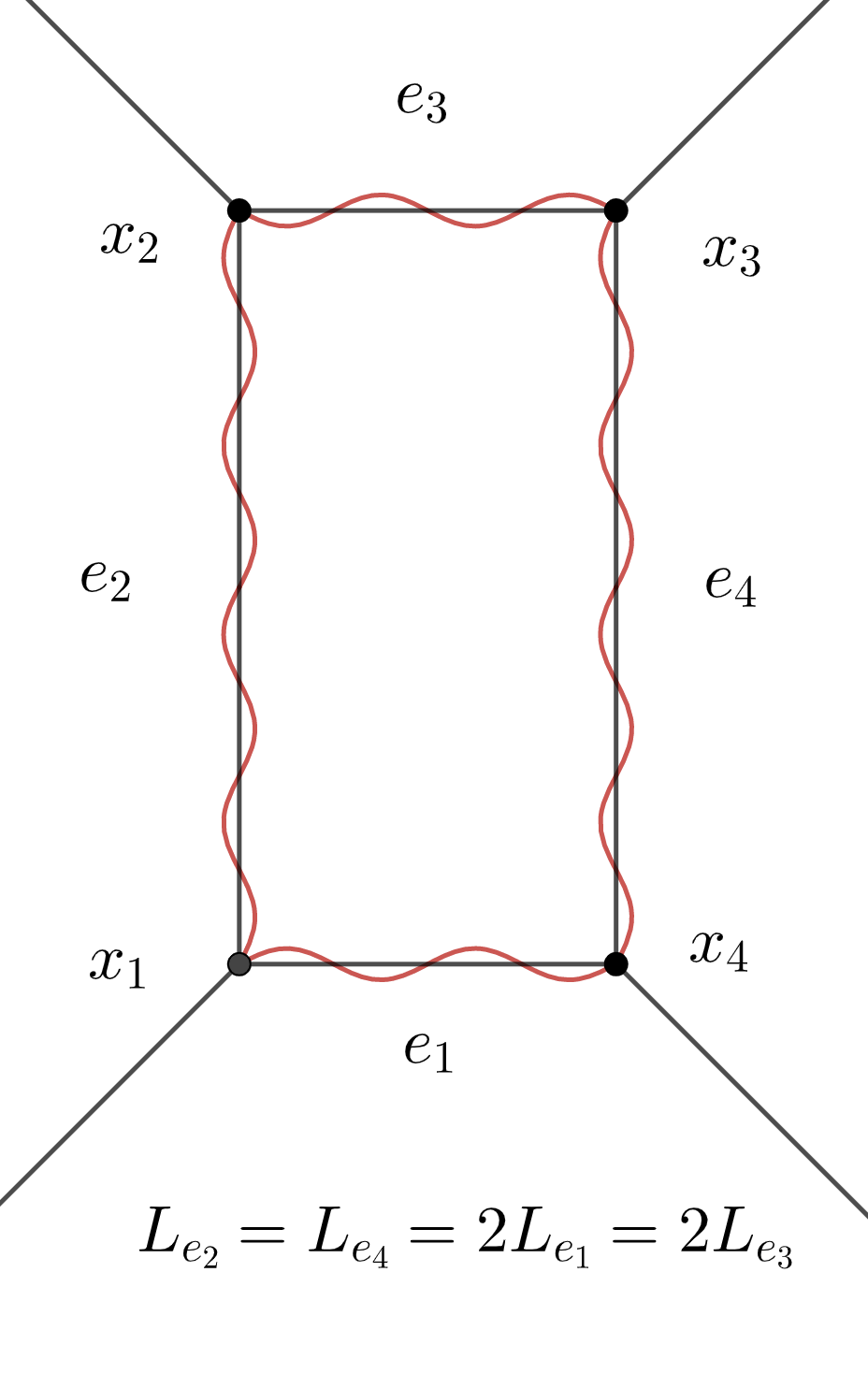}
    \end{minipage}
    \caption{Two  simple examples of quantum graphs having resonances on the real axis. on the left,  if $\lambda (L_{e_1} + L_{e_2} + L_{e_3}) \in \N \pi$, then we may build an eigenfunction with eigenvalue $\lambda^2$, vanishing at $x_1$, and on the lead connected to it. On the right, we may build an eigenfunction vanishing at every vertex, and on all the leads.} \label{exampleTrivial}
\end{figure}

\begin{tcolorbox}
\begin{proposition}
There exists a quantum graph $\mathcal{Q}=(V,E,L, \mathbf{n})$ such that
\begin{itemize}
\item There no loops, i.e., no edges joining a vertex to itself;
\item Each vertex has degree $\geq 3$;
\item The family $(L_e)_{e\in E}$ is rationally independent
\end{itemize}
and $\mathcal{Q}$ has a resonance on the real axis.
\end{proposition}
\end{tcolorbox}

\begin{proof}
The proposition will follow if we show that there exists a closed quantum graph $(V,E,L)$, with no loops, minimal degree $\geq 3$ and  lengths $(L_e)_{e\in E}$ forming a rationally independent family, and having an eigenfunction $\psi$ which vanishes at some vertex $v\in V$. Indeed, in this case, we may connect a lead to $v$, and $\psi$ (extended by zero in the lead) will be a resonant state associated to a resonance on the real axis.

Hence, we will suppose for contradiction, that, for any graph $(V,E)$ with no loops and minimal degree $\geq 3$, we have
\begin{equation}\label{eq:RatDep}
\begin{aligned}
&\left(\text{The closed quantum graph $(V,E,L)$ has an eigenfunction vanishing on a vertex}\right)\\
 &\Longrightarrow \left(\text{ The lenghts $(L_e)_{e\in E}$ are rationally dependent}\right).
\end{aligned}
\end{equation}

To reach a contradiction, we will argue in three steps.

\textbf{Step 1: A surprising rational relation}

Let $(V,E)$ be a quantum graph. Recall that, if we write $E= \{e_1,..., e_J\}$, it is always possible to write the spectrum of $(V,E,L)$ as 
$$\bigcup_{n\in \N} \{\lambda^2_n(L_{e_1}, ..., L_{e_J})\},$$
with each $\lambda_n$ depending in an analytic  way on each variable $L_{e_j}$, $1\leq j \leq J$. See \cite[\S 2.5.1]{BK} (or \cite[Chapter VII, Theorem 3.9]{Kato}  for more details.



Let $(L_{e_j})_{j=1,..., J}$ be a family of rationally independent  lengths, such that $L_{e_j}>2\pi$ for all $1\leq j \leq J$. Let $n_0\in \N$ be such that $ \lambda_{n_0}(L_{e_1}, ..., L_{e_J})\geq 2$. By continuity, we may find $\varepsilon>0$ such that, if $L'_{e_j}\in \left( L_{e_j} - \varepsilon, L_{e_j} + \varepsilon \right)$ for all $j\in \{1,..., J\}$, then $\lambda_{n_0}(L'_{e_1}, ..., L'_{e_J})\geq 1$.

Let $f$ be an eigenfunction associated to $\lambda_{n_0} =  \lambda_{n_0}(L_{e_1}, ..., L_{e_J})\geq 2$, and let $b_0$ be an oriented edge of $(V,E)$ associated to an edge $e_{j_0}\in E$. We know that there exists $a,\phi\in \R$ such that $f_b(x) = a \cos(\lambda_{n_0} x + \phi)$. Since $\lambda_{n_0}\geq 1$ and $L_{e_{j_0}} >2\pi$, the function $f_b$ must vanish at some point $x_b\in (0, \pi)\subset  (0, \frac{L_{e_{j_0}}}{2})$. 

We may add a new vertex at $x_b$, and glue a new graph with six edges at $x_b$, just as in figure \ref{Fig:Gluing}. This new graph contains no loop, and the degrees are all larger than 3. The lengths $\ell_1,..., \ell_6$ are arbitrary.

\begin{figure}[h] 
\center
   \includegraphics[scale=0.4]{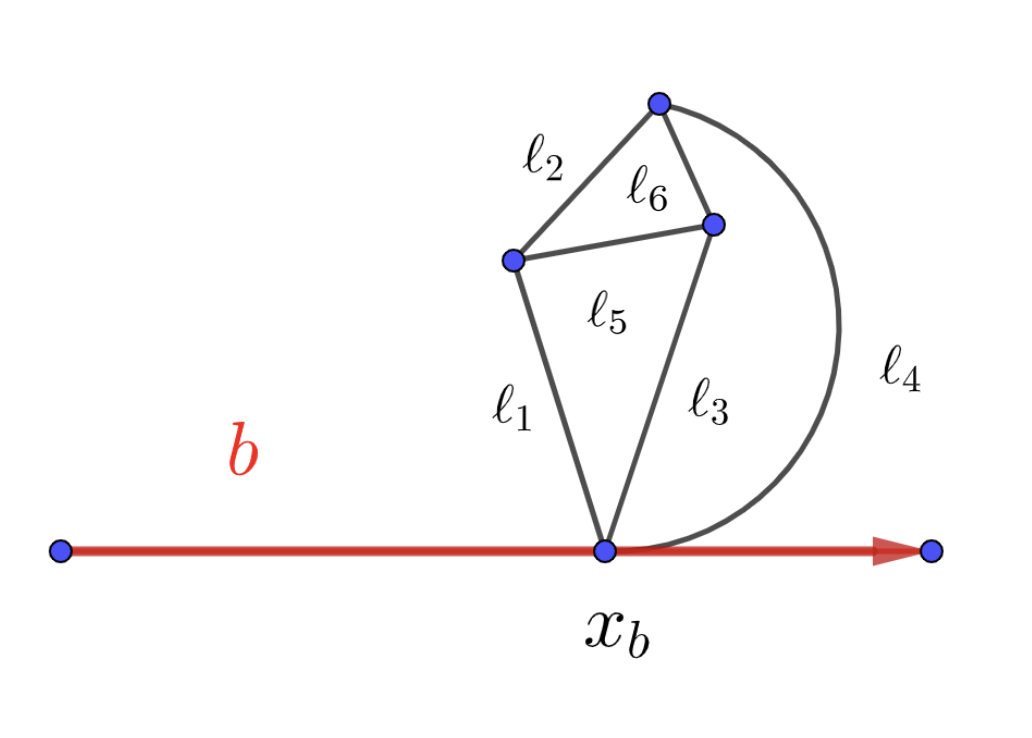}
 \caption{ Gluing a new graph at $x_b$.}\label{Fig:Gluing}
\end{figure}

 Extending the function $\psi$ by zero on the six new edges, we obtain an eigenfunction in this new graph, satisfying the Kirchhoff boundary conditions. We may thus apply (\ref{eq:RatDep}) to deduce that the lengths are rationally dependent. Since the length $\ell_1,..., \ell_6$ are arbitrary and the lengths $(L_{e_j})_{1\leq j \leq J}$ are rationally independent, we deduce that
for each $1\leq j \leq J$, there exists $a_j\in \Q$ such that
$$x_b = \sum_{j=1}^J a_j L_{e_j}.$$

Since we also have $f_b(x_b + \frac{\pi}{\lambda_{n_0}})=0$, we deduce that $x_b+ \frac{\pi}{\lambda_{n_0}}$ is also a rational combination of the $L_e$, and hence, 
\begin{equation}\label{eq:RatDep2}
\frac{\pi}{\lambda_{n_0}} = \sum_{j=1}^J a'_j L_{e_j}
\end{equation}
for some $a_j'\in \Q$.

\textbf{Step 2: Surprising rational relations are everywhere!}

Now, the same argument works for any choice of lengths $(L'_{e_j})_{j=1,..., J}$, provided they are rationally independent and $L'_{e_j}\in \left( L_{e_j} - \varepsilon, L_{e_j} + \varepsilon \right)$ for all $j\in \{1,..., J\}$. We will then obtain equation (\ref{eq:RatDep2}) for these lengths, possibly with other constants $a_j'$, and with $\lambda_{n_0} = \lambda_{n_0}(L'_{e_1}, ..., L'_{e_J})$.

Fix the lengths $(L'_{e_j})_{j=1,..., J-1}$, rationally independent. There are countably many possibilities for the rational numbers $a'_j$, and uncountably many choices for $L'_{e_J}$ to be rationally independent from the other lengths. Hence, there exist coefficients $a'_1,...,a'_J$ such that (\ref{eq:RatDep2}) holds for uncountably many choices of lengths $(L'_{e_J})$. Since $\lambda_{n_0}$ depends in an analytic way on the choice of the lengths, we deduce that, given lengths $(L'_{e_j})_{j=1,..., J-1}$ that are rationally independent, there exist rational numbers $a'_j$ such that (\ref{eq:RatDep2}) holds for all $L'_{e_J}\in \left(0, +\infty \right)$.

By recurrence, we then show that we may find rational numbers $a'_j$ such that (\ref{eq:RatDep2}) holds for all $L'_{e_j}\in \left( 0, +\infty \right)$, $1\leq j\leq J$. 

Hence, we have just shown that there exists $(a_{1,n_0},..., a_{J, n_0})\in \Q^J$ such that for all $(L'_{e_1}, ...,L'_{e_J})\in (0, +\infty)^J$, we have 

\begin{equation}\label{eq:RatDep3}
\frac{\pi}{\lambda_{n_0}} = \sum_{j=1}^J a_{j,n_0} L'_{e_j}.
\end{equation}

In particular, if the graph is equilateral, so that $L'_{e_j}= L$ for all $j$, we have
\begin{equation}\label{eq:Ratio2}
\frac{L \lambda_{n_0}(L,..., L)}{\pi} \in \Q.
\end{equation}

Now, the only thing about $n_0$ we used to prove (\ref{eq:RatDep3}) was the fact that $ \lambda_{n_0}(L_{e_1}, ..., L_{e_J})\geq 2$. But this condition holds for all but finitely many $n_0$. Therefore, for all but finitely many $n\in \N$, we have
\begin{equation}\label{eq:Ratio}
\frac{L \lambda_{n}(L,..., L)}{\pi} \in \Q.
\end{equation}

\textbf{Step 3: Too much rationality is insane !}

The spectrum of an equilateral quantum graph has been known for a long time (see for instance \cite[Theorem 3.6.1]{BK}, which comes from \cite{EquiPank}, and more generally, the references in \cite[Chapter]{BK}).
 If $\Sigma$ denotes the spectrum of the Laplacian\footnote{Given by $(\Delta f)(v) = f(v) - \frac{1}{d(v)} \sum_{w\sim v} f(w)$, where $w\sim v$ if $w$ and $v$ are connected by an edge, and $d(v)$ is the degree of $v$.} of the underlying discrete graph, then any $\lambda^2$ such that $1-\cos(\lambda L) \in \Sigma$ and $\sin(\lambda L)\neq 0$ is an eigenvalue of the equilateral quantum graph.
 
In particular, if the underlying graph is the complete $n$-graph, then  $1-\frac{1}{n}\in \Sigma$. Therefore, there would be infinitely many eigenvalues $\lambda^2$ such that $\cos(\lambda L) = \frac{1}{n}$. But we have shown previously that, for all but finitely many eigenvalues, $\lambda L$ is a rational angle, so that $\cos(\lambda L)\in (\R\backslash \Q) \cup \left\{0, \pm 1, \pm \frac{1}{2}\right\}$, which gives us a contradiction as soon as $n>2$. 
\end{proof}

\providecommand{\bysame}{\leavevmode\hbox to3em{\hrulefill}\thinspace}
\providecommand{\MR}{\relax\ifhmode\unskip\space\fi MR }
\providecommand{\MRhref}[2]{%
  \href{http://www.ams.org/mathscinet-getitem?mr=#1}{#2}
}
\providecommand{\href}[2]{#2}

\end{document}